\pdfoutput=1
\documentclass[11pt,twoside]{article}

\usepackage{amsmath,bm}
\usepackage{amsthm} 
\usepackage[colorlinks=true,
linkcolor=blue,
urlcolor=blue,
citecolor=blue]{hyperref}
\usepackage{tikz}
\usetikzlibrary{arrows,shapes.callouts, positioning, arrows.meta}
\tikzset{
  box/.style={
    rectangle, rounded corners,
    draw=black, thick,
    align=center,
    minimum height=1cm,
    minimum width=4cm
  },
  arrow/.style={
    ->, thick
  }
}
\usepackage{bookmark}
\usepackage{amssymb}
\usepackage{amsfonts}

\usepackage{titlesec}
\usepackage{verbatim}
\usepackage{fullpage}
\usepackage{color,xcolor}
\usepackage{mathrsfs}
\usepackage[normalem]{ulem}
\usepackage{longtable} 
\usepackage{array}

\usepackage[capitalize]{cleveref}

\usepackage[round,sort]{natbib}
\usepackage{dsfont}
\usepackage{graphicx}
\usepackage{caption}
\usepackage{subcaption}
\usepackage{float}
\usepackage{mathtools}
\usepackage{etoolbox}
\usepackage{thmtools}
\usepackage{enumitem}
\usepackage[textsize=tiny]{todonotes}
\usepackage{booktabs}
\usepackage{caption}
\usepackage{multirow}
\usepackage{esvect}
\usepackage[utf8]{inputenc}

\usetikzlibrary{arrows.meta, positioning}
\usepackage{xcolor}

\newtheorem{theorem}{Theorem}
\newtheorem{lemma}{Lemma}

\newtheorem{assumption}{Assumption}
\declaretheoremstyle[headfont=\bf,bodyfont=\normalfont]{ex}
\declaretheorem[style=ex]{example}
\declaretheoremstyle[bodyfont=\normalfont]{rm}
\declaretheorem[style=rm]{remark}

\DeclareMathOperator*{\argmin}{arg\,min}

\def\BB{\mathbb{B}}

\def\AA{\mathbb{A}}
\def\VV{\mathbb{V}}

\def\KK{\mathbb{K}}
\def\II{\mathbb{I}}
\def\YY{\mathbb{Y}}
\def\XX{\mathbb{X}}
\def\NN{\mathbb{N}}

\def\RR{\mathbb{R}}

\def\EE{\mathbb{E}}
\def\SS{\mathbb{S}}
\def\MM{\mathbb{M}}
\def\ZZ{\mathbb{Z}}

\def\cO{\mathcal{O}}

\def\cH{\mathcal H}
\def\cD{\mathcal D}
\def\cE{\mathcal E}
\def\cR{\mathcal R}
\def\cZ{\mathcal Z}

\def\cF{\mathcal F}

\def\cL{\mathcal L}
\def\cB{\mathcal B}

\def\cX{\mathcal{X}}

\def\rI{\mathrm{I}}
\def\rII{\mathrm{II}}

\def\wh{\widehat}
\def\wt{\widetilde}

\def\T{\top}

\def\i{\infty}
\def\tr{{\rm Tr}}
\def\op{{\rm op}}

\def\d{~{\rm d}}

\def\eps{\varepsilon}

\def\two{\text{two}}

\def\l{\left}
\def\r{\right}
\def\b{\big}
\def\B{\Big}

\usepackage{threeparttable} 
\long\def\comment#1{}

\newcommand{\ydsout}[1]{}

\usepackage{stackrel}

\newcommand{\rank}{\ensuremath{{\rm rank}}}

\newcommand{\RKHS}{\ensuremath{\mathscr{F}}}

\newcommand{\StateSpgoodh}[1]{\ensuremath{}}

\newcommand{\StateSpbadh}[1]{\ensuremath{}}

\newcommand{\Data}{\ensuremath{\mathcal{D}}}

\usepackage{upgreek} \newcommand{\distr}{\ensuremath{\upmu}}

\newcommand{\distrdata}{\ensuremath{\bar{\distr}}}

\newcommand{\sdistrdata}{\ensuremath{\sdistr_{\Data}}}

\newcommand{\sdistr}{\ensuremath{\xi}}

\DeclarePairedDelimiterX{\anglep}[1]{(}{)}{#1}
\makeatletter
\newcommand{\@spanstar}[1]{{\rm span}\anglep*{#1}}
\newcommand{\@spannostar}[2][]{{\rm span}\anglep[#1]{#2}}
\newcommand{\Span}{\@ifstar\@spanstar\@spannostar}
\makeatother

\DeclarePairedDelimiterX{\dfun}[2]{(}{)}{#1 \;\delimsize\|\; #2}
\makeatletter
\newcommand{\@trunstar}[2]{\chi^2\dfun*{#1}{#2}}
\newcommand{\@trunnostar}[3][]{\chi^2\dfun[#1]{#2}{#3}}
\newcommand{\chisq}{\@ifstar\@trunstar\@trunnostar}
\makeatother

\DeclarePairedDelimiterX{\inprod}[2]{\langle}{\rangle}{#1, \, #2}

\DeclarePairedDelimiterX{\kulldiv}[2]{(}{)}{#1\;\delimsize\|\;#2}
\makeatletter
\newcommand{\@kullstar}[2]{D_{\text{KL}}\kulldiv*{#1}{#2}}
\newcommand{\@kullnostar}[3][]{D_{\text{KL}}\kulldiv[#1]{#2}{#3}}
\newcommand{\kull}{\@ifstar\@kullstar\@kullnostar}
\makeatother

\makeatletter
\newcommand{\@hilinstar}[2]{\inprod*{#1}{#2}_{\RKHS}}
\newcommand{\@hilinnostar}[3][]{\inprod[#1]{#2}{#3}_{\RKHS}}
\newcommand{\hilin}{\@ifstar\@hilinstar\@hilinnostar}
\makeatother

\makeatletter
\newcommand{\@mudatainstar}[2]{\inprod*{#1}{#2}_{\distrdata}}
\newcommand{\@mudatainnostar}[3][]{\inprod[#1]{#2}{#3}_{\distrdata}}
\newcommand{\mudatain}{\@ifstar\@mudatainstar\@mudatainnostar}
\makeatother

\makeatletter
\newcommand{\@mudatahinstar}[3]{\inprod*{#2}{#3}_{\distrdata}}
\newcommand{\@mudatahinnostar}[4][]{\inprod[#1]{#3}{#4}_{\distrdata}}
\newcommand{\mudatahin}{\@ifstar\@mudatahinstar\@mudatahinnostar}
\makeatother

\DeclarePairedDelimiterX{\defabs}[1]{|}{|}{#1}
\makeatletter
\newcommand{\@absstar}[1]{\defabs*{#1}}
\newcommand{\@absnostar}[2][]{\defabs[#1]{#2}}
\newcommand{\abs}{\@ifstar\@absstar\@absnostar}
\makeatother

\DeclarePairedDelimiterX{\norm}[1]{\|}{\|}{#1}
\makeatletter
\newcommand{\@normstar}[1]{\norm*{#1}_{\RKHS}}
\newcommand{\@normnostar}[2][]{\norm[#1]{#2}_{\RKHS}}
\newcommand{\hilnorm}{\@ifstar\@normstar\@normnostar}
\makeatother

\DeclareFontFamily{U}{matha}{\hyphenchar\font45}
\DeclareFontShape{U}{matha}{m}{n}{
	<-6> matha5 <6-7> matha6 <7-8> matha7
	<8-9> matha8 <9-10> matha9
	<10-12> matha10 <12-> matha12
}{}
\DeclareSymbolFont{matha}{U}{matha}{m}{n}

\DeclareFontFamily{U}{mathx}{\hyphenchar\font45}
\DeclareFontShape{U}{mathx}{m}{n}{
	<-6> mathx5 <6-7> mathx6 <7-8> mathx7
	<8-9> mathx8 <9-10> mathx9
	<10-12> mathx10 <12-> mathx12
}{}
\DeclareSymbolFont{mathx}{U}{mathx}{m}{n}

\DeclareMathDelimiter{\vvvert} {0}{matha}{"7E}{mathx}{"17}%

\DeclarePairedDelimiterX{\opnorm}[1]{\vvvert}{\vvvert}{#1}

\makeatletter
\newcommand{\@hilopnormstar}[1]{\opnorm*{#1}_{\RKHS}}
\newcommand{\@hilopnormnostar}[2][]{\opnorm[#1]{#2}_{\RKHS}}
\newcommand{\hilopnorm}{\@ifstar\@hilopnormstar\@hilopnormnostar}
\makeatother

\makeatletter
\newcommand{\@muopnormstar}[1]{\opnorm*{#1}_{\distr}}
\newcommand{\@muopnormnostar}[2][]{\opnorm[#1]{#2}_{\distr}}
\newcommand{\muopnorm}{\@ifstar\@muopnormstar\@muopnormnostar}
\makeatother

\makeatletter
\newcommand{\@mudataopnormstar}[1]{\opnorm*{#1}_{\distrdata}}
\newcommand{\@mudataopnormnostar}[2][]{\opnorm[#1]{#2}_{\distrdata}}
\newcommand{\mudataopnorm}{\@ifstar\@mudataopnormstar\@mudataopnormnostar}
\makeatother

\makeatletter
\newcommand{\@supnormstar}[1]{\norm*{#1}_{\infty}}
\newcommand{\@supnormnostar}[2][]{\norm[#1]{#2}_{\infty}}
\newcommand{\supnorm}{\@ifstar\@supnormstar\@supnormnostar}
\makeatother

\makeatletter
\newcommand{\@munormstar}[1]{\norm*{#1}_{\distr}}
\newcommand{\@munormnostar}[2][]{\norm[#1]{#2}_{\distr}}
\newcommand{\munorm}{\@ifstar\@munormstar\@munormnostar}
\makeatother

\makeatletter
\newcommand{\@mudatanormstar}[1]{\norm*{#1}_{\distrdata}}
\newcommand{\@mudatanormnostar}[2][]{\norm[#1]{#2}_{\distrdata}}
\newcommand{\mudatanorm}{\@ifstar\@mudatanormstar\@mudatanormnostar}
\makeatother

\makeatletter
\newcommand{\@xidatanormstar}[1]{\norm*{#1}_{\sdistrdata}}
\newcommand{\@xidatanormnostar}[2][]{\norm[#1]{#2}_{\sdistrdata}}
\newcommand{\xidatanorm}{\@ifstar\@xidatanormstar\@xidatanormnostar}
\makeatother

\makeatletter
\newcommand{\@distrnormstar}[2]{\norm*{#1}_{#2}}
\newcommand{\@distrnormnostar}[3][]{\norm[#1]{#2}_{#3}}
\newcommand{\distrnorm}{\@ifstar\@distrnormstar\@distrnormnostar}
\makeatother

\makeatletter
\newcommand{\@psinormstar}[2]{\norm*{#2}_{\psi_{#1}}}
\newcommand{\@psinormnostar}[3][]{\norm[#1]{#3}_{\psi_{#2}}}
\newcommand{\psinorm}{\@ifstar\@psinormstar\@psinormnostar}
\makeatother

\newcommand{\featureh}[1]{\ensuremath{\phi}}

\newcommand{\Lip}{\ensuremath{L}}

\newcommand{\Lipf}[1]{\ensuremath{\Lip_{f}}}

\newenvironment{carlist}
{\begin{list}{$\bullet$}
		{\setlength{\topsep}{0.1in} \setlength{\partopsep}{0in}
			\setlength{\parsep}{0.1in} \setlength{\itemsep}{\parskip}
			\setlength{\leftmargin}{0.15in} \setlength{\rightmargin}{0.08in}
			\setlength{\listparindent}{0in} \setlength{\labelwidth}{0.08in}
			\setlength{\labelsep}{0.1in} \setlength{\itemindent}{0in}}}
	{\end{list}}

\newcommand{\bcar}{\begin{carlist}}
	\newcommand{\ecar}{\end{carlist}}

 \title{Adaptive Kernel Ridge Regression with Linear Structure: Sharp Oracle Inequalities and Minimax Optimality}
 
\author{Xin Bing\thanks{Department of 
Statistical Sciences, University of Toronto.  \texttt{xin.bing@utoronto.ca}}
\hspace{1cm}Chao Wang\thanks{Department of Mathematical and Computational Sciences, University of Toronto.  \texttt{zchao.wang@utoronto.ca}}}

\begin{document}

\maketitle

\begin{abstract}
      Kernel ridge regression (KRR) is a widely used nonparametric method due to its strong theoretical guarantees and computational convenience. However, standard KRR does not distinguish between linear and nonlinear components in the signal, instead applying a single functional regularization to the entire function. This may lead to unnecessary shrinkage of linear structure and consequently suboptimal prediction performance. In this paper, we propose a modified regression procedure that augments KRR with an explicit linear component. The proposed method has the same computational complexity as standard KRR and introduces no additional tuning parameters. Theoretically, we establish a sharp oracle inequality for the proposed estimator and show that it adaptively captures both linear and nonlinear structure, achieving minimax optimal prediction risk under general kernels. Compared with standard KRR, the proposed method improves both the bias and approximation error at the expense of only an additional parametric variance term, which is negligible in low- and moderate-dimensional settings. In high-dimensional regimes, incorporating ridge regularization for the linear component yields a procedure that performs uniformly no worse than KRR. Extensive simulation studies support the theoretical findings.  
\end{abstract}




\section{Introduction}\label{sec:intro}

Kernel ridge regression (KRR) is a canonical tool for nonparametric regression, offering a flexible framework for estimating complex functions through regularization in a Reproducing Kernel Hilbert Space (RKHS). Due to its strong theoretical guarantees and computational convenience, it has been widely used across statistics and machine learning, including many recent applications such as causal inference \citep{singh2024kernel}, image recognition \citep{an2007face}, reinforcement learning \citep{duan2024optimal}, and feature learning \citep{radhakrishnan2024mechanism}.

Let $\cD := \{(Y_i, X_i)\}_{i=1}^n$ denote a dataset of $n$ independent observations, where the response $Y_i \in \RR$ and the feature vector $X_i \in \cX \subseteq \RR^d$  satisfy
\begin{equation}\label{model}
   Y_i=f^*(X_i)+\epsilon_i,
\end{equation}
with $\EE[\epsilon_i \mid X_i] = 0$ and $\EE[\epsilon_i^2 \mid X_i] = \sigma^2$. Let $K : \cX \times \cX \to \RR$ be a kernel function, and denote by $\cH_K$ its associated RKHS, endowed with norm $\|\cdot\|_K$.  KRR seeks a function over $\cH_K$ by solving  
\begin{align}\label{def_krr}
\wt  g   = \argmin _{g \in \cH_K}  ~ \B\{~   \frac{1}{n} \sum_{i=1}^n \left( Y_i - g (X_i) \right)^2 + \lambda  \|g\|_K^2 \B\}, 
\end{align}
where $\lambda>0$ is some regularization parameter that penalizes the complexity of each candidate $g\in \cH_K$. The representer theorem \citep{kimeldorf1971some} ensures that \eqref{def_krr} admits a closed-form solution for computation.  Theoretically, when the regression function $f^*\in \cH_K$, KRR has been well studied, and its prediction error is known to achieve minimax optimal rates under suitable smoothness conditions on $\cH_K$; see, for instance,  \cite{bartlett2005local,caponnetto2007optimal,smale2007learning}. Allowing $f^* \notin \cH_K$ is considered in \cite{steinwart2009optimal,steinwart2008support} and recently by \cite{fischer2020sobolev,bing2025kernel}. In this paper we also allow $f^*\notin \cH_K$.

In many applications, however, the regression function comprises both linear and nonlinear components. Standard KRR does not distinguish between these parts, instead penalizing the entire function through the single RKHS norm in \eqref{def_krr}. As a result, even simple linear signals are subject to regularization, which may introduce unnecessary shrinkage and degrade prediction accuracy. In particular, KRR does not explicitly exploit the parametric structure of the linear component and may therefore estimate it less efficiently than methods tailored to linear models.  An exception arises in classical smoothing spline models, where the linear component lies in the null space of the associated spline penalty, such as the cubic smoothing spline in one dimension or thin-plate spline models in higher dimensions (see, e.g., \cite{Wahba1990,GreenSilverman1994,Wood2017}). As a result, smoothing splines can always capture linear signals, while simultaneously allowing for estimation of more complex smooth components. However, this salient feature is specific to particular spline constructions and does not extend to general kernels.

We illustrate this effect in Figure \ref{fig_intro_new}, where we consider the periodic spline kernels which only include constant functions in their null space as well as the Gaussian kernels which have no nontrivial null space. The first setting is univariate with  
 $f^*(X) =2 X+\alpha \sin (2\pi X)$ where $X \sim \text{Unif}(0,1)$ and the noise follows $N(0,1.5^2)$. The parameter $\alpha$ controls the magnitude of the nonlinear signal, so that the relative strength of the linear component decreases as $\alpha$ increases.  The periodic spline  kernel is $K(x,x')= 1+\sum_{k\in \ZZ\setminus\{0\}}\exp(2\pi i k(x-x'))|k|^{-2q}$ with $q>1/2$    \citep{wahba1990spline}. We consider three choices of $q \in \{0.7, 1, 3\}$, with larger $q$ corresponding to a {more regular} $\cH_K$. The left panel of Figure \ref{fig_intro_new} shows that KRR is clearly outperformed by the Ordinary Least Squares (OLS) predictor in the linear case ($\alpha = 0$). As $\alpha$ increases, KRR begins to outperform OLS. However, it still fails to fully capture the linear signal, as evidenced by the fact that our proposed procedure improves performance across all settings. The second setting considers
$f^*(X) = X^\T\beta+ \alpha [\sin (\pi X_1)+ \cos (\pi X_2X_3)]
$ with $\beta=(2,-1.5,0.5)^\T$  and standard Gaussian noise. Coordinates of  $X$ are generated from $\text{Unif}(-2,2)$ independently. 
For Gaussian kernels $K(x,x') = \exp(-\|x-x'\|^2/\gamma)$, their universality mitigates the previous drawback when $\alpha = 0$ by choosing a larger $\gamma=100$, which corresponds to a simpler $\cH_K$. However, in the presence of nonlinear signal,  $\gamma = 100$ is outperformed by $\gamma =1$ for $\alpha \ge 0.8$, and by $\gamma = 0.5 $ for $\alpha \ge 1$. The smaller choices of $\gamma$ better capture the nonlinear signal but fail to fully recover the linear component. On the other hand, our proposed procedure improves upon KRR for all choices of $\gamma$; for simplicity, we report a single choice (see \cref{sec_sim} for details).

\begin{figure}[ht]
    \centering
\includegraphics[width=0.4\textwidth]{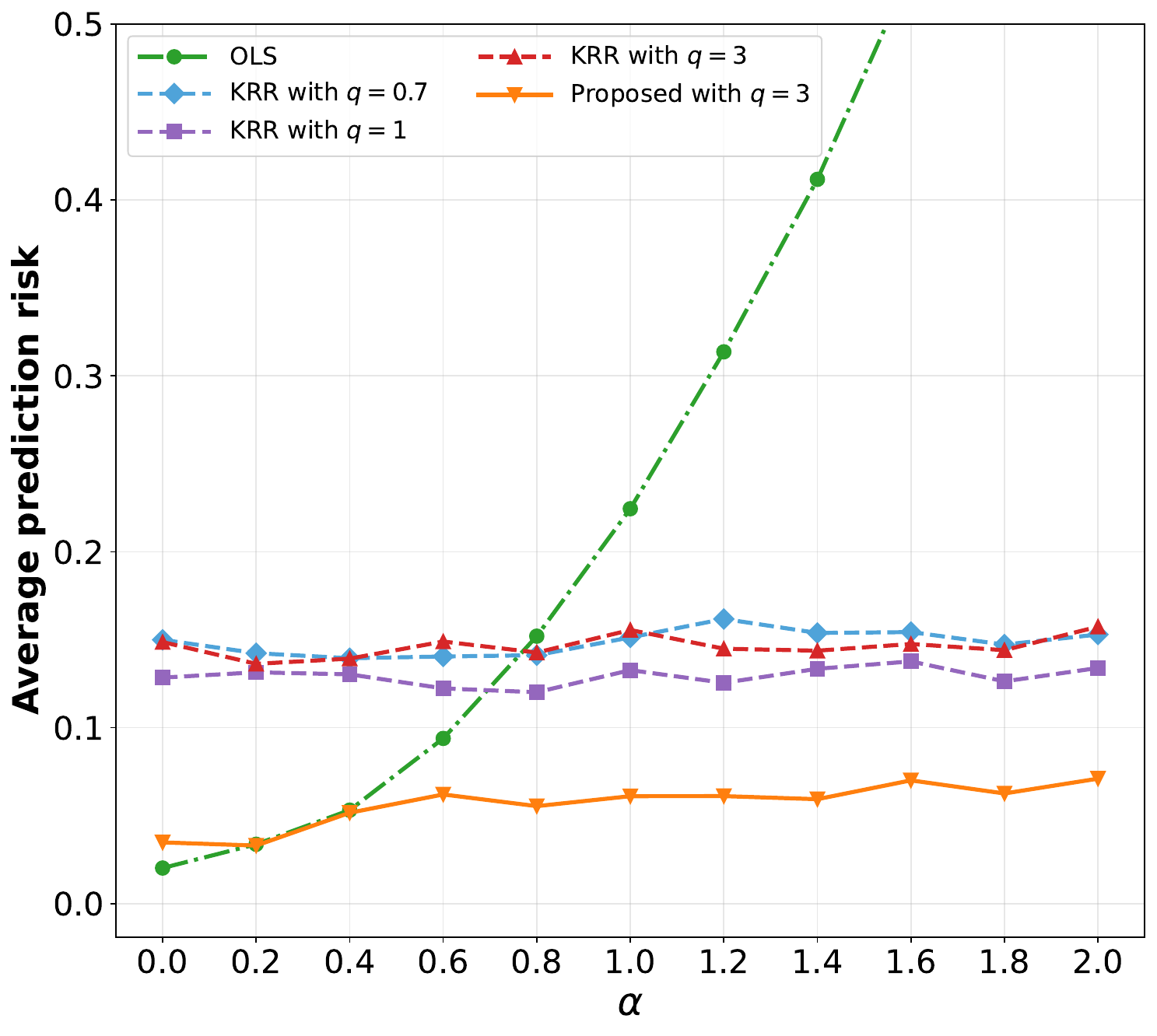}
    \hspace{5mm}
        {\includegraphics[width=0.4\textwidth]{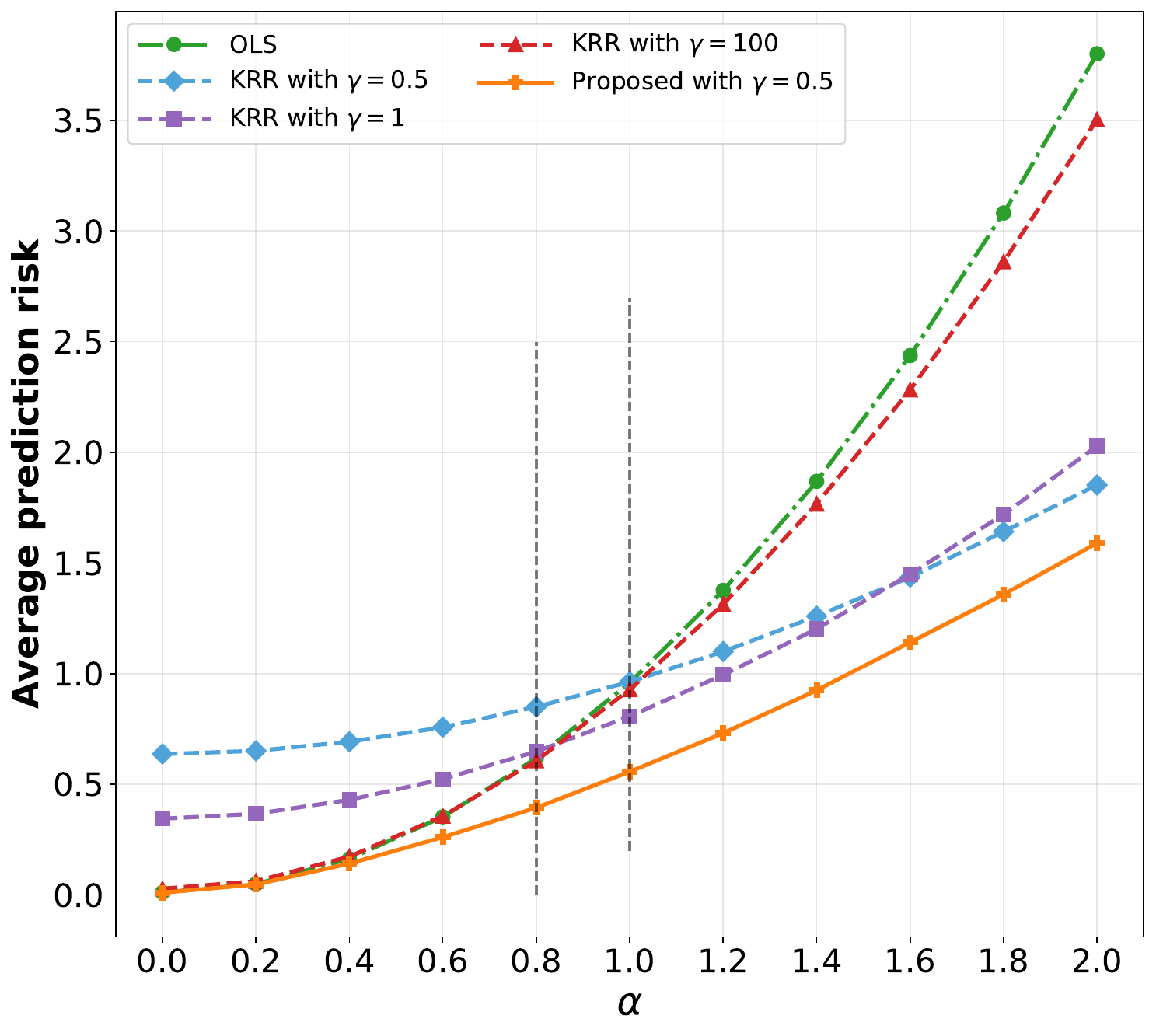}}
\caption{The average  prediction risk against  $\alpha$. Left:  the spline kernel with $n=200$. Right:  the Gaussian kernel with $n=400$.   The tuning parameter $\lambda$ is selected via cross-validation. 
The risk is computed on an independent dataset of size $500$, and each setting is repeated 100 times. 
}
\label{fig_intro_new}
\end{figure}

Motivated by this observation, we propose a modified kernel ridge regression framework that augments the RKHS component with an explicit linear term of the form
\begin{equation}\label{def_f_hat_intro}
\wh f(x) = x^\T \wh \alpha + \wh g(x)\qquad \forall ~ x\in \cX,
\end{equation}
where  $(\wh \alpha, \wh g)$ are jointly optimized by solving 
\begin{align}\label{def_proposed_intro}
     \min_{\alpha\in \RR^d,  g\in \cH_K}   \B\{~  \frac{1}{n} \sum_{i=1}^n\bigl(Y_i- X_i ^\T  \alpha - g(X_i)\bigl)^2 +\lambda\|g\|_K ^2 \B\}.
\end{align}
This formulation allows the linear signal to be estimated without shrinkage from the RKHS penalty, while retaining the flexibility of kernel methods for modeling the nonlinear component. The joint optimization also distinguishes this procedure from a naive two-step algorithm that first regresses $Y$ onto $X$ and then applies standard KRR to the residuals. We refer to \cref{rem_twostep} of \cref{sec_method_krr} for a detailed discussion of such comparison.
In \cref{sec_method_krr} we further show that the optimization \eqref{def_proposed_intro} admits a closed-form solution and can be solved with the same computational complexity as solving the standard KRR.

Our main theoretical contribution is to analyze the prediction risk of the proposed predictor in \cref{sec_theory_proposed} and to demonstrate its advantages over standard KRR.  For ease of presentation, we focus on the in-sample prediction risk under both fixed and random designs, as it already illustrates the main ideas. Similar results can be extended to out-of-sample prediction risk, albeit with more cumbersome analysis. For any measurable function $f: \cX \to \RR$ that may depend on the data $\cD$, we define its (in-sample)  prediction risk as 
\begin{equation}\label{def_risk}
    \cR(f-f^*) :=    {1\over n} \sum_{i=1}^n \EE\l[(f-f^*)(X_i) ^2\r]  =   {1\over n} \sum_{i=1}^n \EE\l[(f(X_i)-f^*(X_i))^2\r].
\end{equation}
To use standard KRR as a benchmark, we first present in \cref{thm_krr_fix} an oracle inequality for its prediction error, a result that, to the best of our knowledge, is not available in the existing literature.
For a fixed $\cH_K$ and regularization parameter $\lambda>0$, \cref{thm_krr_fix} shows that the prediction risk of $\wt g$ in \eqref{def_krr} decomposes into three terms:
\begin{equation}\label{rate_KRR_intro}
   \cR(\wt g-f^*) ~  \le ~   \textsf{Var}_{\cH_K}(\lambda) + \inf_{f\in \cH_K}\Bigl\{\lambda \|f\|_K^2 + \cR(f-f^*)
    \Bigr\} .
\end{equation}
On the right-hand side, the first term represents the variance of $\wt g$, which depends on the complexity of $\cH_K$ and the regularization parameter $\lambda$. The remaining two terms capture the bias induced by regularization and the approximation error arising from restricting the search to $\cH_K$. Notably, KRR effectively selects the function $f \in \cH_K$ that balances bias and approximation, which gives rise to a sharp oracle inequality. As the regularization parameter $\lambda$ increases, we observe the classical bias–variance tradeoff: the variance decreases while the bias increases.

Compared to the above rate, our theory in \cref{thm_fix} for fixed design and \cref{thm_random} for random design shows that the proposed $\wh f$ in \eqref{def_f_hat_intro} satisfies
\begin{equation}\label{rate_proposed_intro}
   \cR(\wh f-f^*)  ~ \le ~ \sigma^2 {d\over n} + \textsf{Var}_{\cH_K}(\lambda) + \inf_{\alpha\in\RR^d, f\in \cH_K}\Bigl\{\lambda \|f\|_K^2  + \cR(\langle \alpha,\rangle + f-f^*)
    \Bigr\} .
\end{equation}
The benefit of incorporating the linear term $x^\T \wh \alpha$ in \eqref{def_f_hat_intro} is evident from the resulting bias and approximation error: by optimally balancing the linear and nonlinear components, the proposed predictor achieves smaller bias and approximation error than KRR. Consequently, it allows for a larger choice of $\lambda$ to better reduce the variance, leading to a smaller overall prediction error. To illustrate this, consider the case where $f^*$ is linear: by setting $f\equiv 0$ and taking $\lambda \to \infty$ in \eqref{rate_proposed_intro}, the proposed estimator $\wh f$ recovers the risk $\sigma^2 d/n$, which holds for general $\cH_K$ and thus extends classical results for specific kernels used in smoothing spline models. We emphasize that this reduction to the parametric rate under a linear model does not generally hold for KRR, even with an optimal choice of $\lambda$ in \eqref{rate_KRR_intro}. On the other hand, when there is no linear signal, a setting favorable to KRR, taking $\alpha = 0$ in \eqref{rate_proposed_intro} shows that the proposed $\wh f$ is no worse than \eqref{rate_KRR_intro}, up to the term $\sigma^2 d/n$. As detailed in \cref{rem_var}, this term can be negligible in general settings with small or moderate $d$. Finally, the advantage of $\wh f$ over $\wt g$ persists more broadly whenever a nontrivial linear component is present in $f^*$; see \cref{rem_correct}.

When $d$ is relatively large, we show in \cref{sec_ridge} that a ridge-type linear estimator can be used in place of OLS to further reduce the term $\sigma^2 d/n$. Although this introduces an additional regularization parameter, optimizing over both the ridge regularization parameter and the RKHS regularization parameter yields a risk bound that adapts to the presence of both linear and nonlinear signals, and is uniformly better than that of standard KRR; see \cref{thm_fix_ridge,bd_pred_rate_random_ridge}.
\\

Developing procedures that adapt to the intrinsic complexity of the underlying signal has long been a central theme in statistical learning. A large body of work on adaptive estimation, model selection, and oracle inequalities has emphasized the importance of procedures that adapt to unknown structural complexity and automatically balance approximation and estimation errors \citep{barron1999risk,birge2001gaussian,tsybakov2009introduction}. In particular, \citet{yang2011parametricness} highlighted the distinction between ``practically parametric'' and ``practically nonparametric'' problems and argued that statistical procedures should adapt to latent low-complexity structure whenever possible. Our goal in this paper, however, is to develop a modification of KRR that adapts to latent linear structure while preserving the flexibility of general RKHS methods.

Our approach is also related to partially linear models (PLS) \citep{HardleLiangGao2000}, in which the regression function is decomposed into a linear component $X_1^\T \beta$ and a nonlinear component $g(X_2)$ where $(X_1,X_2)$ form a partition of the observed features $X$. 
In analyzing PLS, it is often indispensable to assume some identifiability condition to separate the linear and nonlinear signals. For instance, in \cite{muller2015partial} where a joint estimation for both linear and nonlinear components is adopted, a condition that the smallest eigenvalue of  $\EE [ (X_1 -\EE[X_1| X_2])   (X_1 -\EE[X_1| X_2])^\T ] $ is bounded away from zero is required.
A similar treatment was also considered in the subsequent works, see for instance \cite{yu2019minimax}, and recently \cite{lee2024high}.  Alternatively, \cite{zhu2017nonasymptotic,zhu2019high} consider a three-step procedure. First, the response and each coordinate of $X_1$ are separately regressed onto the nonlinear feature $X_2$. Second, the residuals of the response are linearly regressed onto the residuals of the coordinates of $X_1$ to obtain an estimator $\wh \beta$ of $\beta$. Finally, $Y - X_1^\T \wh \beta$ is regressed onto $X_2$ to estimate the nonlinear function $g$.
Their analysis relies on the same identifiability condition
as in \cite{muller2015partial}.  
By contrast, our formulation in \eqref{def_proposed_intro} decomposes the prediction into a linear component and an RKHS component based on the {\em same input} $X$. The identifiability condition commonly used in the PLS literature is clearly not applicable in our context.  Additive models \citep{Wood2017} offer another approach for capturing both linear and nonlinear structure; however, they rely on a coordinate-wise decomposition of the nonlinear component, which can be restrictive in practice.\\

\noindent {\it Notation.} 
For any integer $m$, we let $[m] = \{1,\ldots,m\}$.  
We use $\II_m$  to denote the $m\times m$ identity matrix.
The inner product and the endowed norm in the Euclidean space are denoted as  $\langle\cdot,\cdot\rangle$ and $\|\cdot\|$, respectively.
For any two sequences $a_n$ and $b_n$,  we write $a_n \lesssim b_n$ if there exists some constant $C$ such that $a_n\le C b_n$ for all $n$. We write $a_n\asymp b_n$ if $a_n\lesssim b_n$ and $b_n\lesssim a_n$.  
For any matrix $\MM \in \RR^{d_1\times d_2}$ with $d_1\ge d_2$, we use  $P_ \MM $ to denote the projection matrix onto the column space of $\MM$ and write $Q_{\MM} = \II_{d_1}-P_{\MM}$.

\section{Proposed Method}\label{sec_method_krr}

In this section, we provide details of the proposed method that combines OLS with KRR and discuss its computation.  Stack the responses into   $\YY=(Y_1,\cdots,Y_n)^\T\in \RR^n$ and the features into the design matrix $\XX= (X_1,\cdots , X_n)^\T \in \RR^{n\times d}$. For simplicity, we assume $\rank(\XX)=d$.  For any function $f: \cX \to \RR$, write $f(\XX) =( f(X_1),\cdots,f(X_n))^\T \in \RR^n $. Let $K: \cX \times \cX \to \RR$ be any kernel function that satisfies \cref{ass_pd_K}, and let $\cH_K$ be its corresponding RKHS. Denote by  $\KK$ the  kernel matrix with  $\KK_{ij} = n^{-1} K  (X_i,X_j)$ for $i, j\in [n]$. 

The proposed predictor in \eqref{def_f_hat_intro} requires to solve  
\begin{align}\label{def_proposed_crit}
(\wh \alpha, \wh g) 
=\argmin_{\alpha\in \RR^d,  g\in \cH_K}  \left\{ \frac{1}{n} \|\YY -\XX \alpha - g(\XX)\|^2  +\lambda\|g\|_K ^2 \right\}.
\end{align}
We derive the explicit expression for both $\wh\alpha$ and $\wh g$. Fixing any $g \in \cH_K$, we profile out $\alpha$ by solving 
\[
\wh \alpha_g
=\argmin_{\alpha\in \RR^d}  \left\{ \frac{1}{n} \|\YY -  \XX\alpha - g(\XX)\|^2 
+\lambda\|g\|_{K}^2 \right\}  = (\XX^\T\XX)^{-1}\XX^\T (\YY - g(\XX)).
\]
Substituting this into 
\eqref{def_proposed_crit} yields
\[
\wh g
=\argmin_{ g\in \cH_K}  \left\{  \frac{1}{n}  \b\|Q_{\XX}(\YY -g(\XX) )\b\|^2
+\lambda\|g\|_{K}^2 \right\} = \frac{1}{\sqrt{n}}\sum_{i=1}^n \wh w_i K(X_i,\cdot)
\]
where the second equality is due to the representer theorem \citep{kimeldorf1971some}. The coefficients $\wh w = (\wh w_1,\ldots, \wh w_n)^\T$ are given by 
\begin{align}\label{coeff}
    \wh w & = \argmin_{w \in \RR^n}  \left \{
    {1\over n}\left\| Q_{\XX} (\YY - \sqrt{n} ~ \KK w ) \right\|^2
    +\lambda   w^\T \KK w \right \} = {1\over \sqrt n}(Q_{\XX} \KK+\lambda \II_n)^{-1}Q_{\XX} \YY.
\end{align}
Proof of \eqref{coeff} is given in  \cref{app_pf_sec_2}. 
Consequently, the optimal solution in \eqref{def_proposed_crit} has the expression 
$
    \wh \alpha =(\XX^\T\XX)^{-1}\XX^\T (\YY -  \wh g(\XX)) 
$
and 
$ \wh g(\XX) = \KK \wh w$ 
so that the fit is 
\begin{equation}\label{fit}
    \wh f(\XX) = \XX \wh \alpha + \wh g(\XX) =  P_{\XX}\YY + Q_{\XX}\wh g(\XX) = P_{\XX}\YY +  Q_{\XX}\KK(Q_{\XX} \KK+\lambda \II_n )^{-1}Q_{\XX} \YY.
\end{equation}
 For a new data point $x\in \cX$,  its corresponding response is predicted by
\begin{align} \label{proposed_m}
    \wh f(x) = x^\T\wh \alpha +   \wh g (x) = x^\T  (\XX^\T \XX)^{-1}  \XX ^\T( \YY -    \wh g (\XX) ) +   \wh g (x).
\end{align}

The main computational cost lies in \eqref{coeff}, which is equivalent to solving a standard KRR problem. Therefore, our method neither introduces additional regularization parameters nor incurs extra computational complexity.
Furthermore, our method can be readily integrated with acceleration techniques for KRR, such as random sketch \citep{alaoui2015fast} and random features \citep{rahimi2007random}, to further improve computational efficiency.

\begin{remark}[A two-step procedure]\label{rem_twostep}
    We contrast our proposed procedure with a natural two-step procedure that first regresses the responses onto $\XX$ using OLS, and then adopts KRR on the resulting residuals. Specifically, define $\wh \alpha _{\two}=  (\XX^\T \XX)^{-1}  \XX ^\T \YY$ as the OLS estimator.  
    The  residuals of predicting $\YY$  using $\XX \wh \alpha_{\two}$ are  $  Q_  {\XX} \YY$. Applying KRR to regress $ Q_{\XX} \YY$ onto $\XX$  amounts to solving
    \[
        \wh g_{\two}
        = \argmin_{ g\in \cH_K}  \left\{  \frac{1}{n}  \b\|Q_{\XX}\YY -g(\XX) \b\|^2
        +\lambda\|g\|_{K}^2 \right\},
    \]
    which,  due to the representer theorem, has a solution with its explicit form given in \eqref{def_g_two}.
    The corresponding fit of such two-step procedure is 
    \[
        \wh f_{\two}(\XX) = \XX \wh\alpha_{\two} + \wh g_{\two}(\XX) = P_{\XX} \YY  + \KK (\KK + \lambda \II_n)^{-1}Q_{\XX} \YY.
    \]
    Comparing with the fit in \eqref{fit}, the two-step procedure does not left-project $\KK$ onto $Q_{\XX}$. 
    This distinction can also be understood from an iterative perspective: initializing with $\wh g_0\equiv 0$, at each step $t\ge1$ and given $\wh g_{t-1}$ from the previous iteration, regress the residuals $\YY-\wh g_{t-1}(\XX)$ onto $\XX$ using OLS to obtain
    $
        \wh \alpha_t
        =
        (\XX^\T\XX)^{-1}\XX^\T(\YY-\wh g_{t-1}(\XX)),
    $
    and subsequently apply KRR to regress $\YY-\XX\wh \alpha_t$ onto $\XX$ to obtain $\wh g_t$. If the algorithm stops after $T$ iterations, the resulting predictor is $\wh f_T(\XX)  =  \XX \wh  \alpha_T  + \wh g_T(\XX)$.
    The two-step procedure described above corresponds to terminating after the first iteration ($T=1$). In contrast, the proposed estimator can be interpreted as the fixed-point limit of this iterative procedure, while still admitting the explicit closed-form solution \eqref{proposed_m}. We defer a theoretical comparison between these procedures to \cref{app_comp_two}, while their empirical comparison is presented in \cref{sec_simu_comp_two}.
\end{remark}

\section{Theoretical Analysis}\label{sec_theory_proposed}
In this section, we establish theoretical guarantees for the proposed estimator $\wh f$ in \eqref{proposed_m} in terms of prediction risk, and contrast them with those of the standard KRR estimator $\wt g$ in \eqref{def_krr}. The fixed design setting is analyzed in \cref{sec_theory_fixed}, while the random design setting is studied in \cref{sec_theory_random}.

\subsection{Prediction risk bounds under the fixed design}\label{sec_theory_fixed}

We begin by imposing the following assumption on the kernel function. 
\begin{assumption}\label{ass_pd_K}
    The kernel function $K:{ \cX \times \cX} \to \RR$ is   symmetric and positive semi-definite.\footnote{We say $K$ is positive semi-definite if for all finite sets $\{x_1,\ldots,x_m\}\subset \cZ$ the $m\times m$ matrix  whose $(i,j)$ entry is $K(x_i, x_j)$ is positive semi-definite.} 
\end{assumption}
For any $K$ satisfying \cref{ass_pd_K}, there exists a unique Hilbert space $\cH_K$ on $\cX$ for which $K$ is the reproducing kernel \citep{Aronszajn1950}. That is, for all $x \in \cX$,
\begin{equation}\label{reproducing}
\langle f, K_x \rangle_K = f(x), \qquad \forall~ f \in \cH_K,
\end{equation}
where we write $K_x := K(\cdot, x)$. More specifically, $\cH_K$ is defined as the completion of the linear span of $\{K_x : x \in \cX\}$, equipped with the inner product $\langle \cdot, \cdot \rangle_K$. Our analysis below depends on the $n\times n$ kernel matrix $\KK$ via its eigenvalues, denoted as  $\wh {\mu}_1 \ge \wh {\mu}_2 \ge \cdots\ge \wh {\mu}_n\ge0$.

To illustrate the benefit of the proposed predictor, we start by stating an oracle inequality of the prediction risk of the standard KRR $\wt g$ from \eqref{def_krr}.  Recall its prediction risk $\cR(\wt g-f^*)$ from  \eqref{def_risk}. 

\begin{theorem}\label{thm_krr_fix}
    Grant model (\ref{model}) and
    Assumption   \ref{ass_pd_K}.  For any $\lambda >0$, one has 
    \begin{equation}\label{rate_krr_fix}
        \cR(\wt g-f^*) ~ \le ~    \inf_{f\in \cH_K}\left\{  \cR(f-f^*) + \lambda \|f\|_K^2  \right\}+ {\sigma^2  \over n}  \sum_{i=1}^n\left(\wh \mu_i\over \wh \mu_i + \lambda\right)^2.
    \end{equation}
\end{theorem}
\begin{proof}
    The proof is given in \cref{app_proof_thm_krr_fix}.
\end{proof}

The first term on the right-hand side of \eqref{rate_krr_fix} represents the approximation error due to restricting the search to $\cH_K$. Since its leading constant is one, the oracle inequality in \cref{thm_krr_fix} is sharp. The second term captures the bias introduced by regularization. Together, these two terms reflect how KRR selects an $f \in \cH_K$ that balances approximation error and regularization-induced bias. The third term, on the other hand, corresponds to the variance of $\wt g$, which depends on the kernel matrix $\KK$ and the regularization parameter, but is independent of $f^*$. 
As $\lambda$ increases, the bias increases while the variance decreases. Even in the ideal setting where $f^* \in \cH_K$, achieving fast rates, such as parametric rates, in \eqref{rate_krr_fix} requires $\lambda$ to be large. This, in turn, necessitates a small or even vanishing $\|f^*\|_K^2$, which does not hold in general. This limitation motivates our subsequent analysis of the proposed method. In particular, we show that it attains a significantly smaller approximation error and bias than KRR, while incurring no larger variance up to the term $\sigma^2 d/n$. 

The derivation of the oracle inequality in \cref{thm_krr_fix} is not difficult; however, to the best of our knowledge, it is to some extent a new result for KRR. In the context of ridge regression when $\cH_K$ is finite-dimensional, \citet[Proposition 1]{hsu2014random} provides a bound on $\cR(\wt g - f_{\cH})$ for fixed design, where $f_{\cH}$ attains the minimum of $\cR(f - f^*)$ over $f\in \cH_K$. Most existing oracle inequalities for KRR are formulated for out-of-sample prediction risk; see, for instance, \cite{mourtada2022elementary} and the references therein. However, these results are either not sharp with respect to the leading constant of the irreducible term, or rely on assumptions such as the well-specified model $f^* \in \cH_K$ and the boundedness of the kernel function. \\

We proceed to state the risk bound of the proposed predictor. 
To facilitate understanding, we start by decomposing the regression function $f^*:\cX \to \RR$ as   
 \begin{align}\label{constructive_decom}
     f^*(x) =  x^\T \alpha + (f^*(x) - x^\T \alpha) =: x^\T \alpha + f_\alpha(x)
 \end{align}
where  $\alpha$ is any vector in $\RR^d$. 
The overall risk of predicting $f^*$ thus arises from capturing the linear and nonlinear components separately.  When the signal is entirely linear, i.e., $f_\alpha \equiv 0$ for some $\alpha$, a single OLS predictor yields the prediction risk $\sigma^2 (d/n)$. At the other extreme, when there is no linear signal, i.e., the best choice of $\alpha$ in \eqref{constructive_decom} is zero, using  KRR alone yields the bound in \cref{thm_krr_fix}. In general, an ideal predictor should balance the split between the linear and nonlinear components by choosing $\alpha$ in \eqref{constructive_decom} adaptively. Our proposed predictor achieves this desideratum, as formalized in \cref{thm_fix} below.

It is also worth emphasizing that the decomposition in \eqref{constructive_decom} is introduced only for analytical purposes. In particular, we do not require either $\alpha$ or $f_\alpha$ to be identifiable. Our results hold uniformly over all possible pairs $(\alpha, f_\alpha)$, in particular, under the infimum taken over all such pairs.

Due to the use of the linear predictor in \eqref{def_proposed_crit}, the new prediction risk depends on the $n\times n$ matrix $Q_{\XX} \KK Q_{\XX}$ through its eigenvalues, denoted by $\wh \nu_1 \ge \wh \nu_2 \ge \cdots \ge \wh \nu_n \ge 0$.
 
\begin{theorem}\label{thm_fix}
Grant model (\ref{model}) and
Assumption  \ref{ass_pd_K}. For any $\lambda >0$, one has
\begin{align}\label{bd_pred_rate_fix}
      \cR(\wh f-f^*) \le   \inf_{\alpha \in \RR^d, f\in \cH_K}\left\{  \cR\l(\langle \alpha,\rangle + f-f^*\r) +  \lambda \|f\|_K^2  \right\}+ {\sigma^2\over n}  \left(  d + \sum_{i=1}^n \left(\wh\nu_i\over \wh\nu_i + \lambda\right)^2\right).
\end{align}    
\end{theorem}
\begin{proof}
   The full proof is given in \cref{app_proof_thm_fix}. 
\end{proof}

The risk bound in \eqref{bd_pred_rate_fix} also decomposes into approximation error, regularization-induced bias, and variance. Compared with the variance term in \cref{thm_krr_fix}, the third term in \eqref{bd_pred_rate_fix} consists of both the OLS variance $\sigma^2(d/n)$ and the variance of the fit $\wh g(\XX)$ in \eqref{fit}, with the latter lying in the projected space induced by $Q_{\XX}$. By the min-max characterization of the eigenvalues of $\KK$ and $Q_{\XX} \KK Q_{\XX}$, we prove in \cref{lem_comp} of \cref{app_proof_lemma_thm_fix} that
\begin{equation}\label{var_comp}
  \sum_{i=1}^n  \left(\wh\nu_i \over \wh\nu_i + \lambda\right)^2~ \le  ~  \sum_{i=1}^n  \left(\wh\mu_i \over \wh\mu_i + \lambda \right)^2. 
\end{equation}
Thus, for the same $\lambda$, the variance of $\wh f$ is no greater than that of standard KRR $\wt g$, up to the parametric term $\sigma^2 (d/n)$, which is often negligible, as discussed in the random design setting (see \cref{rem_var}). On the other hand, the proposed predictor can typically accommodate a larger choice of $\lambda$ than KRR due to its smaller approximation error and bias, as detailed below.

The first two terms in \eqref{bd_pred_rate_fix} correspond to the approximation error and the regularization-induced bias of the proposed predictor. 
 Comparing to \cref{thm_krr_fix},  we clearly have  
\begin{align}\label{eq_bd_illu}
    \inf_{\alpha\in\RR^d, f\in \cH_K}\left\{   \cR(\langle \alpha,\rangle + f-f^*)    + \lambda\|f\|_K^2  \right\}  ~ \le ~  \inf_{f\in \cH_K}\left\{   \cR(f-f^*)    + \lambda\|f\|_K^2  \right\},  
\end{align}
implying that incorporating a linear component in $\wh f$ improves the approximation of $f^*$ by optimally balancing the linear and nonlinear parts. Together with \eqref{var_comp}, this implies that for any $\lambda>0$, the risk bound in \cref{thm_fix} converges at least as fast as that in \cref{thm_krr_fix}, up to the additional term $\sigma^2(d/n)$. The benefit of \eqref{eq_bd_illu}, however, can be more pronounced, as illustrated in the remarks below. Technically, establishing the combined bias-approximation term on the left-hand side of \eqref{eq_bd_illu} requires more work because of the left projection $Q_{\XX}$ applied to $\KK$, especially when allowing $f^* \notin \cH_K$ and seeking to establish a sharp oracle inequality.
A key step  is to introduce an intermediate function $f_\lambda$,   defined as the  solution to 
  \[
 \min_{f\in \cH_K} \frac{1}{n} \|Q_{\XX} f(\XX) -Q_{\XX}  f^*(\XX)\|^2  +\lambda\|f\|_K ^2.
  \]
  The bias-approximation term is then related to the $\|Q_{\XX} f_\lambda(\XX) - Q_{\XX} f^*(\XX)\|^2$.

 \begin{remark}[Recover the parametric rate under linear setting]\label{rem_comp_ols}
Consider the case when
$f^*(X) = X ^\T  \alpha^*$ for some $\alpha^*\in \RR^d$.  By choosing $\alpha= \alpha^*$ and $f =0 $ within the infimum in  \eqref{bd_pred_rate_fix}, both the approximation error and the bias are zero, so that for any $\lambda >0$,
\[
\cR(\wh f-f^*)
~ \le ~ {\sigma^2 d\over n}   + {\sigma^2 \over n}\sum_{i=1}^n\left(\wh\nu_i\over \wh\nu_i + \lambda\right)^2 ~\le~     {\sigma^2 d\over n} + {\sigma^2\over n \lambda^2}   \sum_{i=1}^n\wh\nu_i^2.
\] 
Choosing any  $\lambda$ such that $\lambda^2 \gg  d^{-1} \sum_{i=1}^n\wh\nu_i^2 $ gives
$\cR(\wh f-f^*) 
= (1+o(1)) \sigma^2 (d/n) $,
recovering the parametric rate of OLS.  In sharp contrast, the bound of KRR  in \cref{thm_krr_fix} is difficult to achieve such parametric rate for general kernels. 
\end{remark}

The next remark illustrates that the advantage of having smaller approximation and bias terms in \eqref{eq_bd_illu} persists more generally whenever linear signal is partially present.

\begin{remark}[Allow larger $\lambda$ for optimizing the risk bound]\label{rem_lambda}
    From \eqref{eq_bd_illu}, our proposed predictor explicitly captures the linear component in $f^*$ through the infimum over $\alpha$. As a result, the function $f \in \cH_K$ achieving the infimum in the first term of \eqref{bd_pred_rate_fix} typically has a smaller RKHS norm $\|f\|_K^2$. This allows choosing a larger $\lambda$ to reduce the variance term in \eqref{bd_pred_rate_fix}, without requiring a tight trade-off with bias. Consequently, our predictor can attain a smaller prediction risk than KRR. Empirical evidence supporting this behavior is provided in Table \ref{tab_lambda}, where our method consistently selects larger—and in some cases substantially larger—values of $\lambda$ to optimally balance the error terms in the risk bound.
\end{remark}

\subsection{Prediction risk bounds under the random design}\label{sec_theory_random}

The risk bounds in \cref{thm_krr_fix,thm_fix} depend on the design $\XX$ and the kernel matrix $\KK$. In this section we provide deterministic upper bounds when $X_1,\ldots, X_n$ are i.i.d. copies of a random vector $X\in\cX$ from a probability measure $\rho$. Denote by $\cL_2(\rho)$ the space of measurable, square-integrable functions, equipped with the inner product $\langle \cdot , \cdot \rangle_\rho$. We assume $\EE\|X\|^2 < \infty$, so that $\cL_2(\rho)$ contains all linear functions of $X$. 

To quantify the randomness related with $\KK$, we require the kernel function to be continuous, which is standard in the literature and satisfied by many common kernels.
 \begin{assumption}\label{ass_bd_K}
The kernel function $K$ is continuous. 
\end{assumption}

A key element in our analysis is the integral operator  $L_K:  \cL_2(\rho) \to \cL_2(\rho)$, defined  as
\begin{align}\label{def_L_K}
    L_K f:=\int_{\cX} K_{x} f(x)\d\rho(x). 
\end{align}
We assume a standard condition that   $L_K$ can be eigen-decomposed.
\begin{assumption}\label{ass_mercer}
There exist a sequence of eigenvalues $\{\mu_j\}_{j=1}^\infty$ arranged in non-increasing order and corresponding eigenfunctions $\{\phi_j\}_{j=1}^\infty$ that form  an
orthonormal basis of $\cL^2(\rho)$, such that
$
L_{K}   =   \sum_{j=1}^\i \mu_{j}  \langle\phi_{j}, \cdot \rangle_\rho~  \phi_{j}. 
$
\end{assumption}
 \cref{ass_mercer} is guaranteed, for example, by Mercer's theorem \citep{mercer1909xvi} when $\cX$ is compact and \cref{ass_bd_K} holds, though it can also hold in more general settings (see, e.g., \citet{steinwart2012mercer}).  
 Under Assumption \ref{ass_mercer},  
the kernel function $K$ admits the representation
\begin{align}\label{eq_eigen_decomp}
K(x,x')=\sum_{j=1}^\infty\mu_j\phi_j(x)\phi_j(x'), \qquad  \forall ~ x,x'\in\cX. 
\end{align}
In the case $\mu_j>0$ for all $j\ge 1$, $\cH_K$ can be equivalently  written as 
$
\cH_K=  \{f=\sum_{j=1}^\infty\gamma_j\phi_j:
\ \sum_{j=1}^\infty {\gamma_j^2 }/{\mu_j}
<\infty  \}.
$
If for some $k\in \NN$, $\mu_j > 0$ for $j\le k$ and $\mu_j=0$ for all $j>k$, $\cH_K$ reduces to a $k$-dimensional function space spanned by  $\{\phi_1,\ldots, \phi_k\}$.  In this paper, we restrict attention to nondegenerate kernels, that is, $\mu_1>0$. Extending the results to the degenerate case is straightforward.

In view of the risk bounds in \cref{thm_krr_fix,thm_fix} under the fixed design, our analysis requires controlling the variance terms involving the empirical eigenvalues $\{\wh \mu_j\}_{j \in [n]}$ and $\{\wh \nu_j\}_{j \in [n]}$ in terms of the population eigenvalues $\{\mu_j\}_{j \ge 1}$. The latter are in turn linked to the complexity of $\cH_K$, commonly quantified by the kernel complexity function defined as
\begin{equation}\label{kernel_complexity}
      R(\lambda) = \biggl({1\over n}\sum_{j=1}^\i \min\{\lambda,\mu_j\}\biggr)^{1/2},\qquad \forall~  \lambda \ge 0.
\end{equation}
The following theorem provides non-asymptotic risk bounds for the proposed predictor $\wh f$ in \eqref{def_proposed_crit} under the random design setting.
  
\begin{theorem}\label{thm_random}
Grant model (\ref{model}) and
Assumptions  \ref{ass_pd_K} -- \ref{ass_mercer}. For any $\lambda >0$, one has
\begin{align}\label{bd_pred_rate}
 \cR(\wh f-f^*) \le   \inf_{\alpha \in \RR^d, f\in \cH_K}\left\{  \cR\l(\langle \alpha,\rangle + f-f^*\r) +  \lambda \|f\|_K^2  \right\}+  \sigma^2   \left(  \frac{d}{n}+\frac{R(\lambda)^2}{\lambda}  \right).
\end{align}   
\end{theorem}
\begin{proof}
    The proof bounds the variances in \eqref{var_comp} by $R(\lambda)^2 /\lambda$, and is given in \cref{app_proof_thm_random}.
\end{proof}

To explain the bound in \eqref{bd_pred_rate}, we first note that  for any given $\alpha\in\RR^d$ and $f\in\cH_K$,
\[
    \cR\l(\langle \alpha,\rangle + f-f^*\r) = \EE[f(X)+X^\T \alpha - f^*(X)]^2  =: \|f + \langle \alpha, \rangle - f^*\|_\rho^2
\]
where $X \sim \rho$. Let $\Pi_X$ denote the $\cL_2(\rho)$ projection operator onto the linear span of $X$, and $\Pi_X^\perp$ denote its orthogonal complement. The infimum of the bound in \eqref{bd_pred_rate} gets simplified to 
\[
    \inf_{f\in\cH_K}  ~  \|\Pi_X^\perp (f -f^*)\|_\rho^2 + \lambda \|f\|_K^2 
\]
To get more transparent expression, we consider two scenarios in the next two remarks.  

\begin{remark}[Explicit bounds under well-specified settings]\label{rem_correct}
When $\Pi_X^\perp f^* \in \cH_K$,  one has 
\[
    \cR(\wh f-f^*) ~ \le~     \lambda \|\Pi_X^\perp f^*\|_K^2  + \sigma^2 \frac{R(\lambda)^2}{\lambda} + \sigma^2  {d\over n},\qquad \forall~  \lambda >0.
\]
To find the optimal $\lambda$ that minimizes the right-hand side, we derive two bounds. Since $R(\lambda)$ in \eqref{kernel_complexity} is a sub-root function of $\lambda$,\footnote{A function $\psi: [0,\i)\to  [0,\i)$ is a sub-root function if it is nonnegative, non-decreasing, and if $\delta\mapsto \psi(\delta)/\sqrt{\delta}$ is non-increasing for $\delta>0$.} it admits a unique positive solution to $R(\lambda)=\lambda$ \citep[Lemma 3.2]{bartlett2005local}. This solution, denoted by $\delta_n$, is known as the {\em critical radius}. Using the property of sub-root functions (see, \cref{lem_subroot_monotone} of  \cref{app_proof_thm_random})  gives
\begin{equation}\label{bound_simp_1}
    \cR(\wh f-f^*) ~ \le~     \left( \|\Pi_X^\perp f^*\|_K^2  + \sigma^2 \right) \lambda  + \sigma^2  {d\over n},\qquad \forall ~ \lambda \ge \delta_n.
\end{equation}
Alternatively, using $\sum_{j\ge 1} \mu_j \le \EE[K(X,X)]  =: \kappa^2$ from \eqref{eq_eigen_decomp}, it is easy to see that 
\begin{equation}\label{bound_simp_2}
    \cR(\wh f-f^*) ~ \le~      \lambda \|\Pi_X^\perp f^*\|_K^2  + \sigma^2 {\kappa^2 \over n\lambda} + \sigma^2  {d\over n},\qquad \forall~  \lambda >0.
\end{equation}
Optimizing \eqref{bound_simp_2} over $\lambda$, in conjunction with  \eqref{bound_simp_1}, gives 
\begin{equation}\label{bound_simp}
     \cR(\wh f-f^*) ~ \le~  \sigma^2  {d\over n}  +   \l\{\begin{array}{ll}
          \displaystyle  ( \|\Pi_X^\perp f^*\|_K^2  + \sigma^2 )\delta_n   &   \text{if}~~    {1 \over 2\sigma}\|\Pi_X^\perp f^*\|_K \in \l[{\delta_n\sqrt{n}\over \kappa}, ~ {\kappa\over \delta_n\sqrt n}\r]\vspace{2mm}\\
        \displaystyle 2\sigma \|\Pi_X^\perp f^*\|_K  {\kappa \over \sqrt n} \qquad &  \text{otherwise.}
     \end{array}\r.  
\end{equation}
By the same arguments, one can show that, under the more restrictive assumption that $f^* \in \cH_K$, \eqref{bound_simp} also holds for the standard KRR with $\Pi_X^\perp f^*$ replaced by $f^*$ and without the term $\sigma^2 d/n$.  The advantage of the proposed predictor becomes pronounced when $\|\Pi_X^\perp f^*\|_K \ll \|f^*\|_K$. For instance, if  $f^*$ is close to linear in the sense that  $\|\Pi_X^\perp f^*\|_K =o(d/\sqrt n)$, then the second bound in \eqref{bound_simp} converges to $\sigma^2 d/n$ under an optimal choice of $\lambda$ tending to $\infty$, echoing the discussion in \cref{rem_lambda}. On the other hand, $\|f^*\|_K = o(d/\sqrt{n})$ rarely occurs, even when $f^*$ is exactly linear, indicating that the rate of KRR cannot, in general, reduce to $\cO(d/n)$. When $f^*$ contains significant nonlinear components, we typically end up with the first bound in \eqref{bound_simp} whence the critical radius $\delta_n$ plays an important role. We discuss its rate momentarily. 
\end{remark}

\begin{remark}[Explicit bounds with universal kernels on compact sets]\label{rem_bd_appro}
When $\cX$ is compact and the kernel function $K$ is ``universal'' \citep{micchelli2006universal}, provided that  
$\Pi_X^\perp f^*$ is continuous on $\cX$, for any $\eps >0$, there exists $f_\eps\in\cH_K$ such that
\[
    \cR(\wh f-f^*) ~ \le~      \eps + \lambda  \|f_\eps\|_K^2  + \sigma^2 \frac{R(\lambda)^2}{\lambda} + \sigma^2  {d\over n},\qquad \forall~  \lambda >0.
\]
Similar bound as \eqref{bound_simp} can be obtained with $f_\eps$ in place of $\Pi_X^\perp f^*$. 
\end{remark}
 
\begin{remark}[Minimax optimality]
The minimax optimality relies on the notion of statistical dimension, which is defined as the largest index $j$ for which the kernel eigenvalues $\mu_j$
   exceeds $\delta_n$, that is, $d(\delta_n) = \max\{j\ge 0:   \mu_j \ge \delta_n\}$ with $\mu_0:=\i$.  Any kernel whose eigenvalues satisfy $\sum_{j=1}^{d(\delta_n)} \mu_j \lesssim d(\delta_n)\delta_n$ is said to be {\em regular}, a notion introduced in \cite{yang2017randomized}.
   When $f^*$ belongs to the function class $\cB_{\cH_K}:= \{f \in \cH_K : \|f\|_K^2 \le 1\}$  and the inducing kernel is 
   regular with $d(\delta_n)\ge 128\log 2$, the standard KRR achieves the minimax optimal prediction risk bound of order $\cO(\delta_n)$ (see,  for instance, \citet[Theorem 2]{yang2017randomized} for the fixed design and the adaptation of \citet[Theorem 2]{bing2025kernel} for the random design). To discuss the minimax optimality of \eqref{bound_simp},   consider  the larger function space 
    $$
        \cF := \l\{f\in \cL_2(\rho): \exists~  \alpha\in\RR^d, g\in\cB_{\cH_K}, ~ \text{s.t. } f = \langle \alpha, \rangle + g\r\}.
    $$
    Clearly $\cB_{\cH_K} \cup \{\langle \alpha,\rangle: \alpha \in \RR^d\} \subseteq \cF$. Since the term $\sigma^2 d/n$ corresponds to the standard minimax optimal rates under linear models, combined with the above lower bound $\delta_n$ for $f^*\in \cB_{\cH_K}$, one can deduce that the risk bound $\cR(\wh f-f^*) = \cO(d/n + \delta_n)$ achieves the minimax optimal rate for $f^*\in \cF$, implying the minimax optimality of the proposed predictor. 
\end{remark}

\begin{remark}[No worse than the standard KRR]\label{rem_var}
In \cref{rem_correct}, we have illustrated the benefit of the proposed predictor $\wh f$ relative to the standard KRR $\wt g$ when $f^*$ has a significant linear signal. In this remark, we show that even in the setting $f^* \in \cH_K$, which is favorable to $\wt g$ , the proposed predictor has no worse rate.
Suppose $f^* \in \cB_{\cH_K}$ and $d(\delta_n) \ge 128 \log 2$, in which case $\wt g$ is minimax optimal with prediction risk of order $\delta_n$. Regarding $\wh f$, since one can take $\alpha = 0$ in \eqref{bd_pred_rate}, the argument used to prove \eqref{bound_simp_1} shows that the risk bound of $\wh f$ is at most $(1+\sigma^2)\delta_n + \sigma^2 d/n$. The term $\sigma^2 d/n$, however, can be negligible in general settings.

 For instance, $d(\delta_n) \ge 128 \log 2$ implies $\delta_n \ge c/n$ for some constant $c>0$. As a result, when $d = \cO(1)$, one has $d/n = \cO(\delta_n)$. The condition $\delta_n \ge c/n$ also holds when $\mu_1 \ge c'/n$ for some constant $c'>0$. See \cref{lem_lb_delta_n} in \cref{app_proof_thm_random} for the proof. We emphasize that $\mu_1 \ge c'/n$ is much weaker than the commonly assumed $\mu_1 > c'$ in the existing literature \citep{caponnetto2007optimal,rudi2017generalization}.

 For certain specific kernels, we can further derive explicit expressions for $\delta_n$  (see, for instance, the proof in \citet{bing2025kernel}). Concretely, for polynomial kernels whose eigenvalues of the integral operator $L_K$ satisfy
    $\mu_j \asymp  j^{-2\beta}$ for all $j\ge 1$ with $\beta >1/2$,  we have  
    $
        \delta_n ~ \asymp ~  n^{-{2\beta /(2\beta+1)}}
    $
    to which $\sigma^2 d/n$  becomes negligible when $d = \cO(n^{1/(2\beta+1)})$. Common polynomial kernels include the spline kernel \citep{wahba1990spline} and kernels inducing Besov RKHSs \citep{fischer2020sobolev}. 
A  well-known case with polynomial decaying eigenvalues is the    Sobolev RKHS, with the integer $s>d/2$ being its associated 
smoothness parameter. For such a case, one has  $\mu_j \asymp  j^{-2s/d}$ for all $j\ge 1$  \citep{edmunds1996function} under mild conditions, 
so that
$\sigma^2 d/n$  becomes negligible when $d = \cO(n^{d/(2s+d)})$.  
    For  exponential kernels whose eigenvalues of $L_K$ satisfy $\mu_j \asymp \exp(-\gamma j)$ for all $j\ge 1$  with some constant $\gamma>0$,  we have  $ \delta_n\asymp  \log (n) / n $ to which $\sigma^2 d/n$  becomes negligible whenever $d = \cO(\log n)$. Such kernels include the Gaussian kernel.
\end{remark} 

Although our method is more favorable than KRR in low- or moderate-dimensional settings, the price term $\sigma^2 d/n$ can be large when $d$ is large, due to the use of OLS. In the following section, we replace OLS with ridge regression to address this issue.

\section{Adaptive Kernel Ridge Regression with Ridge Penalty}\label{sec_ridge}

When $d$ is moderate or large, it is necessary to incorporate shrinkage into the proposed method in \cref{sec_method_krr} to better capture the linear component. To this end, we replace OLS with ridge regression and propose to solve 
\begin{align}\label{def_method_ridge}
(\wh \alpha_r, \wh g_r ) 
~ =~ \argmin_{\alpha\in \RR^d,  g \in \cH_K}  \left\{ \frac{1}{n} \sum_{i=1}^n\l (Y_i- X_i ^\T  \alpha - g(X_i)\r)^2 + \mu \|\alpha\|^2  +\lambda \|g\|_K^2 \right\},
\end{align}
where  $\mu\ge 0$ is  an additional parameter to regularize $\|\alpha\|^2$.
Similar to that in \cref{sec_method_krr}, both $\wh \alpha_r$ and 
$\wh g_r $ in \eqref{def_method_ridge} have closed form:
\begin{align} \label{closed_form_ridge}
     \wh \alpha_r = ( \XX ^\T \XX+ n \mu  ~ \II_d) ^{-1}  \XX^\T (\YY -  \wh g_r(\XX)), \qquad 
\wh g_r
~ =~ { 1\over \sqrt{n}} \sum_{i=1}^n \wh \beta_i ~K(X_i,\cdot),
\end{align}
with $\wh \beta  = (1/ \sqrt n)(Q_\mu   \KK+ \lambda \II_n)^{-1} Q_\mu \YY.$ Here  for any $\mu\ge 0$, we write $P_\mu =\XX (\XX^\T\XX + n \mu~  \II_d )^{-1} \XX^\T $ and 
$ Q_\mu =\II_n - P_\mu  $.
We defer detailed derivation of \eqref{closed_form_ridge} to \cref{app_pf_closed_form_ridge}.
Finally,  the fit  is 
\begin{equation}\label{fit_ridge}
    \wh f_r(\XX) :  = \XX \wh \alpha_r + \wh g_r(\XX) = P_{\mu} \YY +  Q_\mu \KK( Q_\mu \KK +\lambda \II_n )^{-1} Q_\mu \YY.
\end{equation}  
When $\mu = 0$ in \eqref{def_method_ridge}, 
the above $\wh f_r$ reduces to $\wh f$ in \eqref{fit}.  
When $\mu \to \infty$, $\wh f_r$ reduces to the standard KRR in \eqref{def_krr}.
On the other hand, as $\lambda\to\infty$, $\wh f_r$ reduces to the ridge regression.

Although we focus on the ridge regression to handle the high-dimensional issue, the proposed framework can be extended to other variants tailored to different model structures.
For instance, when  the linear part in \eqref{constructive_decom} exists sparse structure,    the ridge penalty in \eqref{def_method_ridge} 
can be replaced by lasso-type penalty \citep{tibshirani1996regression},  penalty used in  adaptive lasso \citep{zou2006adaptive}, or others. 
When  the nonlinear part in \eqref{constructive_decom} has dependence only on  some low-dimensional  
representation of $\XX$, the inputs of $g(\cdot)$ in \eqref{def_method_ridge} can be replaced with their predicted values  \citep{bing2025kernel}.\\

To analyze the prediction risk of $\wh f_r$, we can drop the assumption $\rank(\XX)=d$.
Write the eigenvalues of $\wh \Sigma : =  (1/n) \XX^\T\XX $ as $\wh \tau_1\ge\wh \tau_2\ge\cdots\ge \wh \tau_d\ge0$, and the eigenvalues of $Q_\mu \KK Q_\mu$ as $\wt \nu_1 \ge \wt \nu_2 \ge \cdots \ge \wt \nu_n \ge 0$.  
The following theorem states the oracle inequality of prediction risk bound for 
$\wh f_r$ in \eqref{fit_ridge}, with dependence on  the eigenvalues of both $ \wh \Sigma  $ and  $Q_\mu \KK Q_\mu$.

\begin{theorem}\label{thm_fix_ridge}
Grant model (\ref{model}) and
Assumption  \ref{ass_pd_K}. 
For any $\mu, \lambda >0$, one has
\begin{equation}
    \label{bd_pred_rate_fix_ridge}
   \begin{aligned}
  \cR( \wh  f_r -f^*)  ~ \le~    & \inf_{\alpha \in \RR^d, f\in \cH_K}\left\{  \cR\l(\langle \alpha,\rangle + f-f^* \r)+\mu\|\alpha\|^2 +  \lambda \|f\|_K^2  \right\} \\
  & \quad + {2\sigma^2 \over n}  \left( \sum_{i=1}^d \left(\wh\tau_i\over \wh\tau_i  + \mu \right)^2  + \sum_{i=1}^n \left(\wt \nu_i\over \wt\nu_i + \lambda\right)^2\right).
\end{aligned}     
\end{equation}
\end{theorem}
\begin{proof}
    Its proof can be found in \cref{app_pf_thm_ridge}.
\end{proof}

The risk bound in \eqref{bd_pred_rate_fix_ridge} consists of the same components as that in \eqref{bd_pred_rate_fix}.
The differences are the additional ridge-induced bias $\mu\|\alpha\|^2$, the reduced variance of the ridge estimator $\XX \wh \alpha_r$ compared to OLS, and the replacement of the eigenvalues ${\wh \nu_i}$ by ${\wt \nu_i}$.
The benefit of using ridge over OLS is evident for large $d$, as increasing $\mu$ reduces variance by trading off bias in capturing the linear component. Regarding the variance of the nonlinear predictor, since
 the comparison inequality  \eqref{var_comp} also holds when $\wh \nu_i$'s are replaced by $\wt \nu_i$'s by following the same argument, the overall variance of $\wh f_r$ is no greater than that of KRR, up to the term $(\sigma^2/n)\sum_i \wh\tau_i^2/(\wh\tau_i+\mu)^2$, which can be made arbitrarily small by taking $\mu$ sufficiently large. Meanwhile, note that the approximation error and bias of $\wh f_r$ are no greater than, and can be substantially smaller than, those for KRR. Thus, for a suitable choice of $\mu$, the proposed predictor is never worse than KRR across all settings and can yield clear improvements over KRR.

To further quantify the randomness in the variance term in the bound \eqref{bd_pred_rate_fix_ridge}, we consider the random design setting and establish the corresponding prediction risk bound.  
Write $\Sigma = \EE XX^\T \in \RR^{d\times d}$ with its eigenvalues  $\tau_1\ge\tau_2\ge \cdots\ge \tau_d\ge 0$. 
The following theorem states the upper bound of the prediction risk of $\wh f_r$, in terms of both $R(\lambda)$ defined in \eqref{kernel_complexity} and the eigenvalues of $\Sigma$. 

\begin{theorem}\label{bd_pred_rate_random_ridge}
Grant model (\ref{model}) and
Assumptions  \ref{ass_pd_K} -- \ref{ass_mercer}. For any  $\mu, \lambda >0$, one has
\begin{align*}
     \cR( \wh  f_r -f^*)  \le    \inf_{\alpha \in \RR^d, f\in \cH_K}\left\{  \cR\l(\langle \alpha,\rangle + f-f^* \r)+\mu\|\alpha\|^2 +  \lambda \|f\|_K^2  \right\} + \sigma^2  \left( {1 \over n}  \sum_{i=1}^d {\tau_i\over \tau_i + \mu}   + \frac{R(\lambda)^2}{\lambda}\right).
\end{align*}
\end{theorem}
\begin{proof}
    Its proof can be found in \cref{app_pf_thm_random_ridge}.
\end{proof}

Comparing to \cref{thm_random}, the risk bound in \cref{bd_pred_rate_random_ridge} additionally depends on  $\Sigma$ through its eigenvalues as well as the regularization parameter $\mu$. 
 To facilitate understanding, suppose there exist optimal $ \alpha_*$ and $f_*$ that attain the infimum of the approximation error and bias terms. Figure \ref{fig_change_mechanism} illustrates the effect of increasing the ridge penalty $\mu$. This may occur, for instance, when the proportion of linear signal in $f^*$ decreases, or when either the eigenvalues of $\Sigma$ or the dimension $d$ increase. As shown in Figure \ref{fig_change_mechanism}, a larger $\mu$ leads to a smaller $\|\alpha_*\|^2$ in order to offset the additional bias induced by ridge regularization. Consequently, a larger portion of the signal in $f^*$ must be captured by the nonlinear predictor. This, in turn, necessitates a smaller choice of $\lambda$, at the expense of a larger variance term $\sigma^2 R(\lambda)^2/\lambda$.
A similar line of reasoning applies to a larger choice of the RKHS penalty $\lambda$, which may arise when the proportion of nonlinear signal decreases or when the complexity of the kernel becomes larger. In summary, the optimal choices of $\mu$ and $\lambda$ tend to move in opposite directions in order to achieve the best decomposition between the linear and nonlinear components. In our numerical studies below, we find that cross-validation yields satisfactory performance.

\begin{figure}[ht]
\centering
\begin{tikzpicture}[
    scale=0.68,
    transform shape,
    box/.style={
        rectangle,
        rounded corners=6pt,
        draw=black,
        line width=0.8pt,
        align=center,
        font=\small
    },
    mainbox/.style={
        box,
        minimum width=2.75cm,
        minimum height=1.15cm
    },
    bluebox/.style={
        box,
        draw=blue!70!black,
        fill=blue!3,
        line width=0.9pt,
        minimum width=3.2cm,
        minimum height=1.45cm
    },
    orangebox/.style={
        box,
        draw=orange!85!black,
        fill=orange!5,
        line width=0.9pt,
        minimum width=3.2cm,
        minimum height=1.45cm
    },
    varbluebox/.style={
        box,
        draw=blue!70!black,
        fill=blue!3,
        line width=0.9pt,
        minimum width=3.75cm,
        minimum height=1.85cm
    },
    varorangebox/.style={
        box,
        draw=orange!85!black,
        fill=orange!5,
        line width=0.9pt,
        minimum width=3.75cm,
        minimum height=1.85cm
    },
    bluearrow/.style={
        -{Stealth[length=2.8mm]},
        blue!75!black,
        line width=0.9pt,
        shorten <=8pt,
        shorten >=8pt
    },
    orangearrow/.style={
        -{Stealth[length=2.8mm]},
        orange!85!black,
        line width=0.9pt,
        shorten <=8pt,
        shorten >=8pt
    }
]

\node[mainbox] (mu) at (-1.2,0)
{Increase $\mu$};

\node[bluebox] (lin) at (3.3,1.35)
{\textbf{Optimal linear}\\[-1mm]
 \textbf{strength}\\[1mm]
 $\|\alpha_* \|^2$\\[1mm]
 {\color{blue!75!black}decreases}};

\node[orangebox] (nonlin) at (8.3,1.35)
{\textbf{Optimal nonlinear}\\[-1mm]
 \textbf{strength}\\[1mm]
 $\|f_* \|_K^2$\\[1mm]
 {\color{orange!85!black}increases}};

\node[bluebox] (lambda) at (13.2,1.35)
{\textbf{Optimal $\lambda$}\\[1mm]
 \\[1mm]
 {\color{blue!75!black}decreases}};

\node[varbluebox] (ridgevar) at (3.3,-1.75)
{\textbf{Variance of linear}\\[-1mm]
 \textbf{ridge estimator}\\[1mm]
 $\displaystyle \frac{\sigma^2}{n}
\sum_{i=1}^d \frac{\tau_i}{\tau_i+\mu}$\\[1mm]
 {\color{blue!75!black}decreases}};

\node[varorangebox] (krrvar) at (13.2,-2.)
{\textbf{Variance of KRR}\\[-1mm]
 \textbf{estimator}\\[1mm]
 $\displaystyle \sigma^2
 \frac{R(\lambda)^2}{\lambda}$\\[1mm]
 {\color{orange!85!black}increases}};

\draw[bluearrow] (mu.east) -- (lin.west);
\draw[bluearrow] (mu.east) -- (ridgevar.west);

\draw[orangearrow] (lin.east) -- (nonlin.west);
\draw[orangearrow] (nonlin.east) -- (lambda.west);

\draw[bluearrow] (lambda.south) -- (krrvar.north);

\end{tikzpicture}
\caption{Effect of a larger $\mu$ on the risk bound in \cref{bd_pred_rate_random_ridge}.}
\label{fig_change_mechanism}
\end{figure}
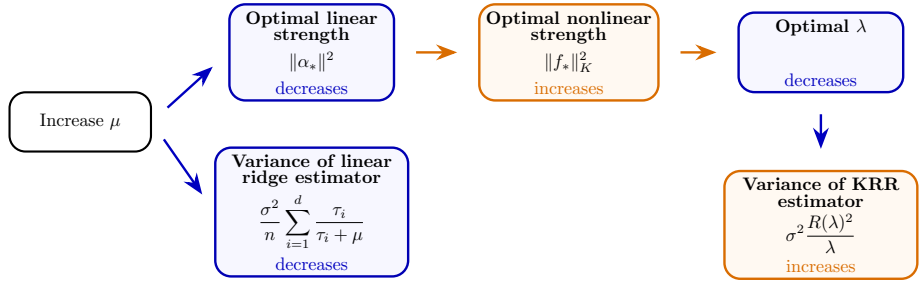

Finally, under the linear model $f^*(x) = x^\top \alpha^*$, the risk bound in \cref{bd_pred_rate_random_ridge} can be upper bounded by
$
              \mu \|\alpha^*\|^2  + {2\sigma^2 \tr(\Sigma) / (\mu n)}
$
which is a standard risk bound for ridge regression \citep[Remark 15]{hsu2014random}.

\section{Simulation Studies}\label{sec_sim}
In this section,  we extend simulation studies in the Introduction to comprehensively evaluate the empirical performance of the proposed algorithm in \cref{sec_method_krr}.  
For comparison, we consider two competitors: KRR and  OLS. 
 For kernel-based methods, we consider
 two different types of 
 kernel: the spline kernel
$K(x,x')= 1+\sum_{k\in \ZZ\setminus\{0\}}\exp(2\pi i k(x-x'))|k|^{-2q}$   with $q>1/2$ in \cref{sec_spline};   the Gaussian kernel $K(x,x') =\exp(-\|x-x'\| ^2/\gamma)$ with $\gamma>0$ in \cref{sec_gaussian}. 
The tuning parameter $\lambda$ is always selected via cross-validation (CV). For each time, the performance for each predictor is evaluated by the prediction risk computed on an independent test dataset with size $500$. 
Each setting is repeated  100 times, and average prediction risks are reported. In \cref{app_sec_sim_ridge}, we report additional simulation results on the ridge-based predictor in \cref{sec_ridge}.

\subsection{Spline kernel}\label{sec_spline}
As in the Introduction, the model \eqref{model} is set to  
 $f^*(X) =2 X+\alpha \sin (2\pi X)$, where
 $\epsilon\sim N(0,1.5^2)$,  $X \sim \text{Unif}(0,1)$ and $\alpha $ controls the proportion of the nonlinear component. For kernel-based methods, we apply the spline kernel by choosing $q$ in $\{0.7,1,3\}$.  
 We first increase the sample size $n$ to compare the empirical convergence rates of different predictors, with  multiple choices of $\alpha$ in 
 $\{0, 0.5,1,1.5\}$. The obtained results are reported in Figure \ref{simu_results}. Table 
 \ref{tab_slope} summarizes
 the corresponding absolute slopes in Figure \ref{simu_results}.  
 Table \ref{tab_lambda} summarizes
 the CV-selected optimal $\lambda$ with $n=400$.
 \begin{figure}[ht]
\centering
\includegraphics[width=0.85 \textwidth]{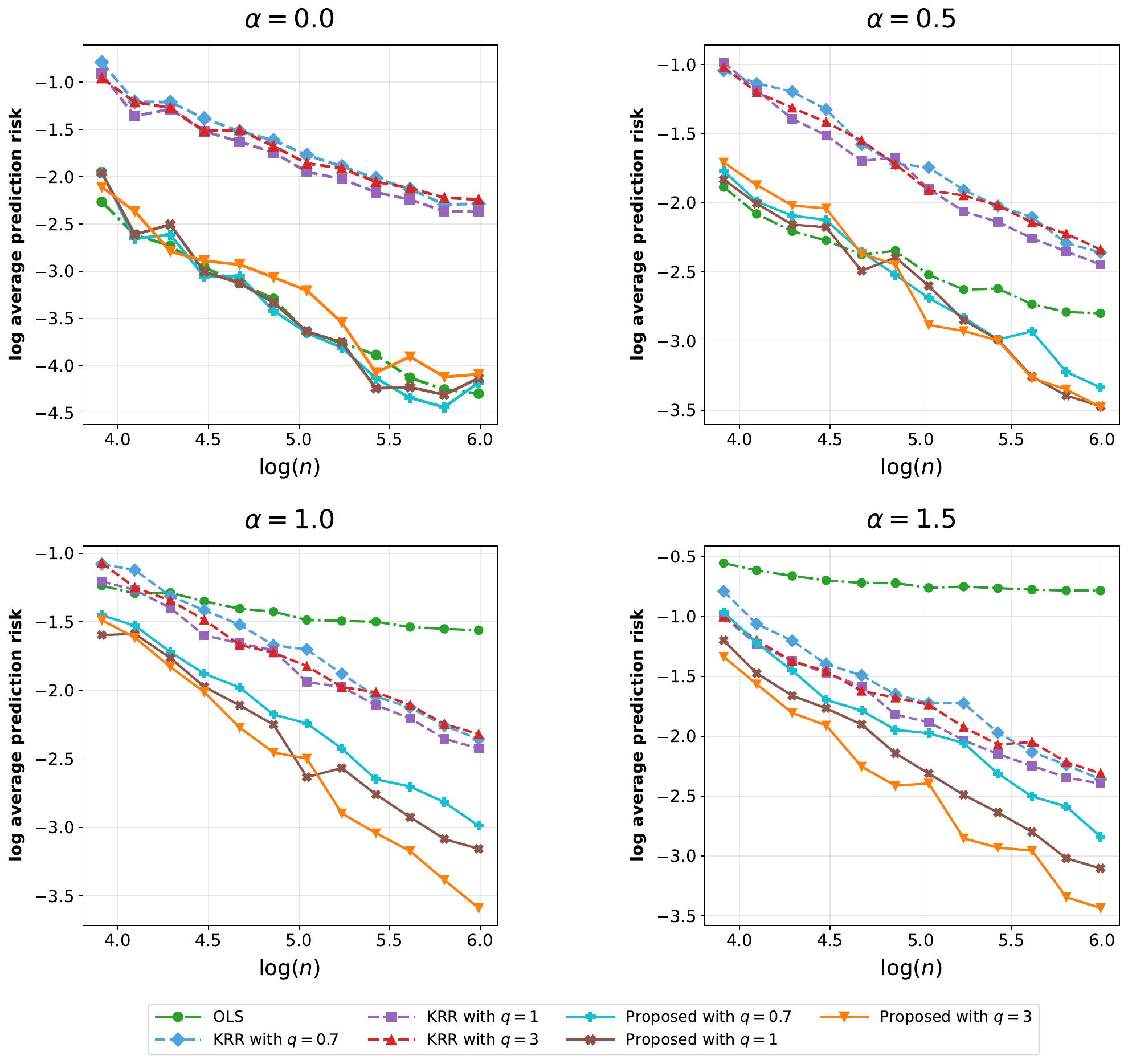}
    \caption{The log average  prediction risk against  $\log n$.}
    \label{simu_results}
\end{figure}

\begin{table}
\centering
\caption{Absolute slopes of the predictors with varying $\alpha$.}
\small
\setlength{\tabcolsep}{7pt}
\begin{tabular}{lccccccc}
\toprule
& OLS & \multicolumn{3}{c}{KRR} & \multicolumn{3}{c}{Proposed} \\
\cmidrule(lr){2-2} \cmidrule(lr){3-5} \cmidrule(lr){6-8}
$q$ in spline kernel & -- & 0.7 & 1 & 3 & 0.7 & 1 & 3 \\
\midrule
$\alpha = 0$ & 0.992 & 0.686 & 0.671 & 0.615 & \textit{1.117} & 1.101 & 0.984 \\
$\alpha = 0.5$ & 0.421 & 0.660 & 0.685 & 0.627 & 0.741 & 0.805 & \textit{0.894} \\
$\alpha = 1$   & 0.162 & 0.632 & 0.602 & 0.586 & 0.751 & 0.821 & \textit{1.026} \\
$\alpha = 1.5$ & 0.098 & 0.702 & 0.670 & 0.597 & 0.817 & 0.908 & \textit{0.999} \\
\bottomrule
\end{tabular}
\label{tab_slope}
\end{table}

\begin{table}
\centering
\caption{$\log \lambda$ of the proposed method and KRR with varying $\alpha$ ($n=400$).}
\label{tab_lambda}
\small
\setlength{\tabcolsep}{7pt}
\begin{tabular}{lcccccc}
\toprule
& \multicolumn{3}{c}{KRR} & \multicolumn{3}{c}{Proposed} \\
\cmidrule(lr){2-4} \cmidrule(lr){5-7}
$q$ in spline kernel & 0.7 & 1 & 3 & 0.7 & 1 & 3 \\
\midrule
$\alpha = 0$   & -2.13 & -2.75 & -5.41 & 3.86 & 3.82 & 3.82 \\
$\alpha = 0.5$ & -2.17 & -2.71 & -4.45 & -0.63 & -0.95 & -1.40 \\
$\alpha = 1$   & -2.11 & -2.51 & -4.86 & -1.33 & -1.71 & -2.54 \\
$\alpha = 1.5$ & -2.15 & -2.79 & -5.78 & -1.93 & -2.18 & -3.18 \\
\bottomrule
\end{tabular}
\end{table}

 As expected, it is observed from Figure \ref{simu_results} and Table 
\ref{tab_slope} that the proposed methods for all values of $q$ outperform both KRR and OLS in almost all settings, aligning with the theoretical analysis in \cref{sec_theory_proposed}.
Comparing the slopes of the $\log$ error declines,
we observe that in the linear setting when $\alpha =0$,  OLS exhibits an almost optimal convergence rate of  $1/n$, whereas all the KRR estimators decay at slower rates. By contrast, our method fully captures the linear signal, and has 
comparable, even slightly better in some cases, predictive performance to OLS.  This observation is aligned with  \cref{rem_comp_ols}.
When $\alpha$  increases, the performance of OLS deteriorates from moderate to severe extent due to under-fitting. KRR, on the other hand, has better performance than OLS in these settings, as it is better able to accommodate the nonlinear structure. Nevertheless, our method still has clear advantages over KRR, especially for relatively small $\alpha$. This is because KRR does not exploit the linear structure, whereas our method adapts to the linear signal, thereby achieving superior prediction performance.

More interestingly,  when $\alpha>0$,  our method achieves increasingly better performance as  $q$  increases, exhibiting a trend opposite to that observed for KRR.  Recall that a larger $q$ corresponds to a more regular RKHS $\cH_K$. 
Specifically, for our method,   decreasing the richness of the RKHS reduces the variance, and this gain outweighs the accompanying increase in bias and approximation error, a benefit from our method incorporating the explicit linear component, as shown in \eqref{eq_bd_illu}.
Conversely, for KRR,  a more regular RKHS  causes a substantially more adverse effect on the bias and approximation error,  as KRR fails to capture the linear signal. 
Regarding the CV-selected optimal $\lambda$ reported in Table \ref{tab_lambda}, we observe that the selected $\lambda$ for our method is generally larger than that for KRR, which corroborates  \cref{rem_lambda}. 
Furthermore, the selected $\lambda$ decreases as $q$ increases in each setting,  except for our method with $\alpha=0$. This is due to the fact that a larger $q$ leads to a reduced complexity of $\cH_K$, so that less regularization is needed to control the variance.
 The exception arises because, in the linear setting, our method does not require additional shrinkage induced by the RKHS penalty.

The second experiment is to illustrate the effect of $\alpha$ on predictive performance, where we fix $n=200$. The results in Figure \ref{fig_vary_alpha_spline} complement those reported in the Introduction. One additional observation is that, as $\alpha$ increases, the proposed method tends to favor a larger choice of $q$.
\begin{figure}[ht]
    \centering

        {\includegraphics[width=0.45\textwidth]{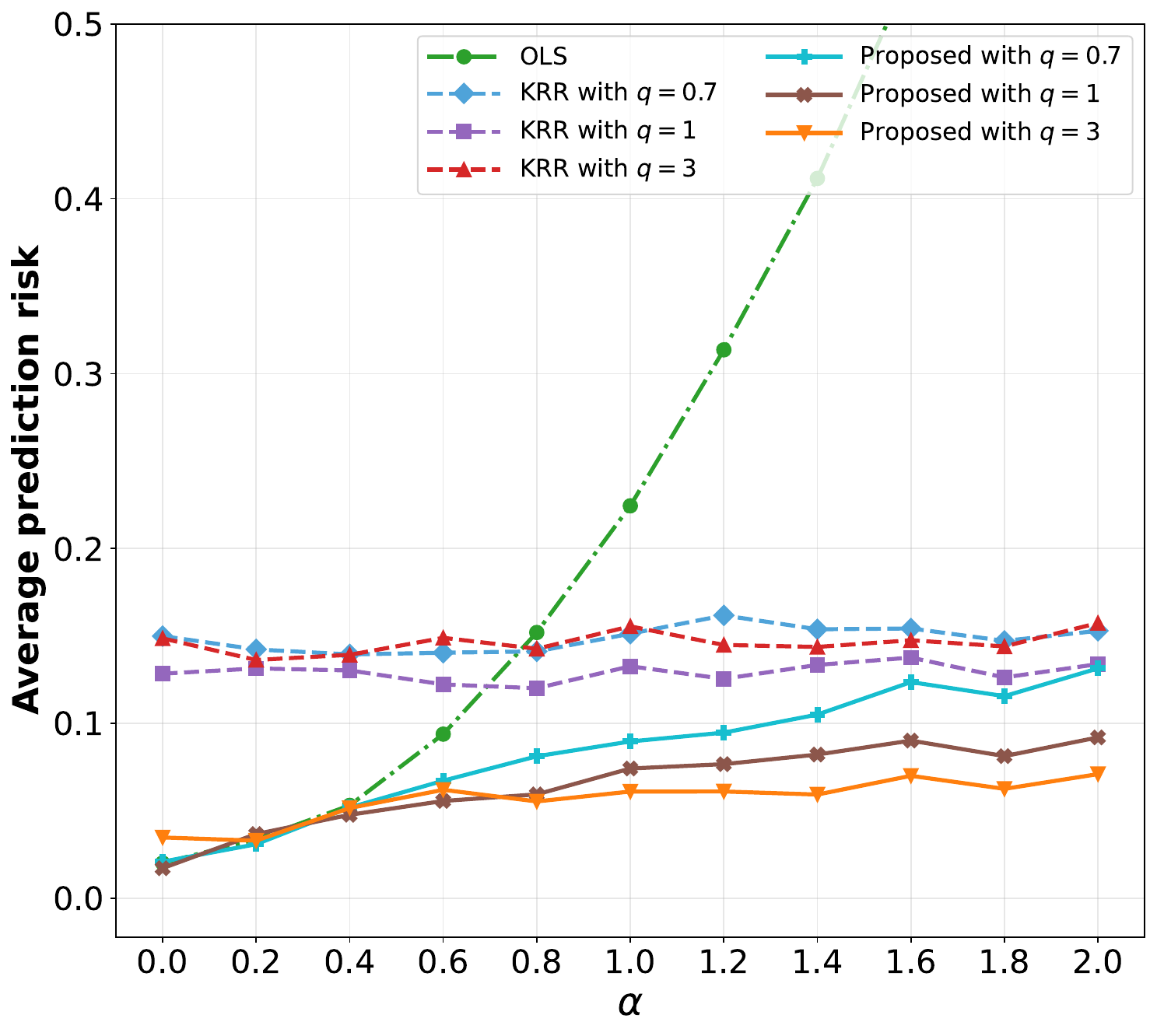}}
              
\caption{
The average prediction risk against $\alpha$. 
}
\label{fig_vary_alpha_spline}
\end{figure}

\subsection{Gaussian kernel}  \label{sec_gaussian}

Recall from the Introduction that for the Gaussian kernel,  the model \eqref{model} is set to    
$f^*(X) = X^\T\beta+ \alpha [\sin (\pi X_1)+ \cos (\pi X_2 X_3)]
$, where   $\beta=(2,-1.5,0.5)^\T$,   $\epsilon\sim N(0,1)$ and the coordinates of $X$ are generated  from $\text{Unif}(-2,2)$ independently.  In this experiment, we examine how the performance of different predictors evolves when varying  $\alpha$.  The sample size is set to $n=400$.

We first fix 
the bandwidth $\gamma$ in  $\{ 0.5, 1, 100\}$. 
A smaller $\gamma$ corresponds to a richer  $\cH_K$.
The results in the left panel of 
Figure \ref{fig_vary_alpha_gaus} completes those reported in the Introduction. 
We observe that the predictive performance of both KRR and our method is affected by  $\gamma$. 
The RKHS $\cH_K$ induced by the Gaussian kernel with  $\gamma=100$  is relatively simple and behaves like a linear function space, so that the performance of both KRR and our method with $\gamma=100$ align closely with that of OLS. However, the RKHS with large $\gamma$ has limited capacity of approximation to nonlinear functions, so that both KRR and our method
have increasingly inferior performance as $\alpha$ grows.  
A small $\gamma$ helps them capture the nonlinear signal, leading to better performance for both. However, KRR with a smaller $\gamma$ is difficult to capture the linear signal, leading to substantially worse performance than OLS when $\alpha$ is relatively small. By contrast, our method still effectively adapts to the linear signal, attributed to the explicit incorporation of a linear predictor.  Comparing the results for $\gamma =0.5$ and $\gamma =1$, the advantage of our method over KRR remains visible even when $\alpha$ is large.

It is clear that the bandwidth of the Gaussian kernel plays a crucial role in prediction.  
In practice, selecting a proper bandwidth is difficult. 
A natural treatment is to tune both $\lambda$ and bandwidth simultaneously, even though it is shown by our experiments that our method is relatively robust to the variations of $\gamma$, except when $\gamma =100$, an extremely large $\gamma$. 
To this end, 
we choose $12$ candidates of 
$\gamma$ values,  equally spaced on the log scale between 0.1 and 150, and select  $\gamma$ jointly with $\lambda$ by CV.  Additionally, we also 
follow the suggestion in \cite{mukherjee2010learning} to select a median-adjusted (Med) $\gamma$, where the used $\gamma$ is  0.2 of the squared median of all pairwise distances among sample points.
The average prediction risks are reported in the right panel of  Figure \ref{fig_vary_alpha_gaus}. KRR with median-adjusted $\gamma$ is clearly inferior to KRR with CV-selected $\gamma$. Even the latter is still inefficient when $\alpha$ is relatively large, since
it does not balance well between capturing the linear and nonlinear signals. More precisely, KRR with a relatively small $\gamma$ is likely to miss much of the linear signal, while   KRR with a relatively large $\gamma$ tends to limit its ability to capture the nonlinear signal. 
In sharp contrast, 
it can be observed that our method with CV-selected $\gamma$ performs best across all values of $\alpha$.  Our method with median-adjusted $\gamma$ also has close performance to that with   CV-selected $\gamma$.
If the computational time is limited, the empirical results recommend using our method with median-adjusted $\gamma$.

 Practically, we have no knowledge in advance whether the linear signal or the nonlinear signal in the underlying regression function is dominant, making it difficult to determine whether to use a linear procedure or the more complex KRR method. The theoretical analysis in \cref{sec_theory_proposed}, together with empirical results, suggests that regardless of the strength of linear and nonlinear signals, 
 the proposed method should always be preferred in low- or moderate-dimensional settings. For high-dimensional settings, a variant stated in 
 \cref{sec_ridge} can be adopted to adapt to the feature dimension. 
\begin{figure}[ht]
    \centering
        {\includegraphics[width=0.395\textwidth]{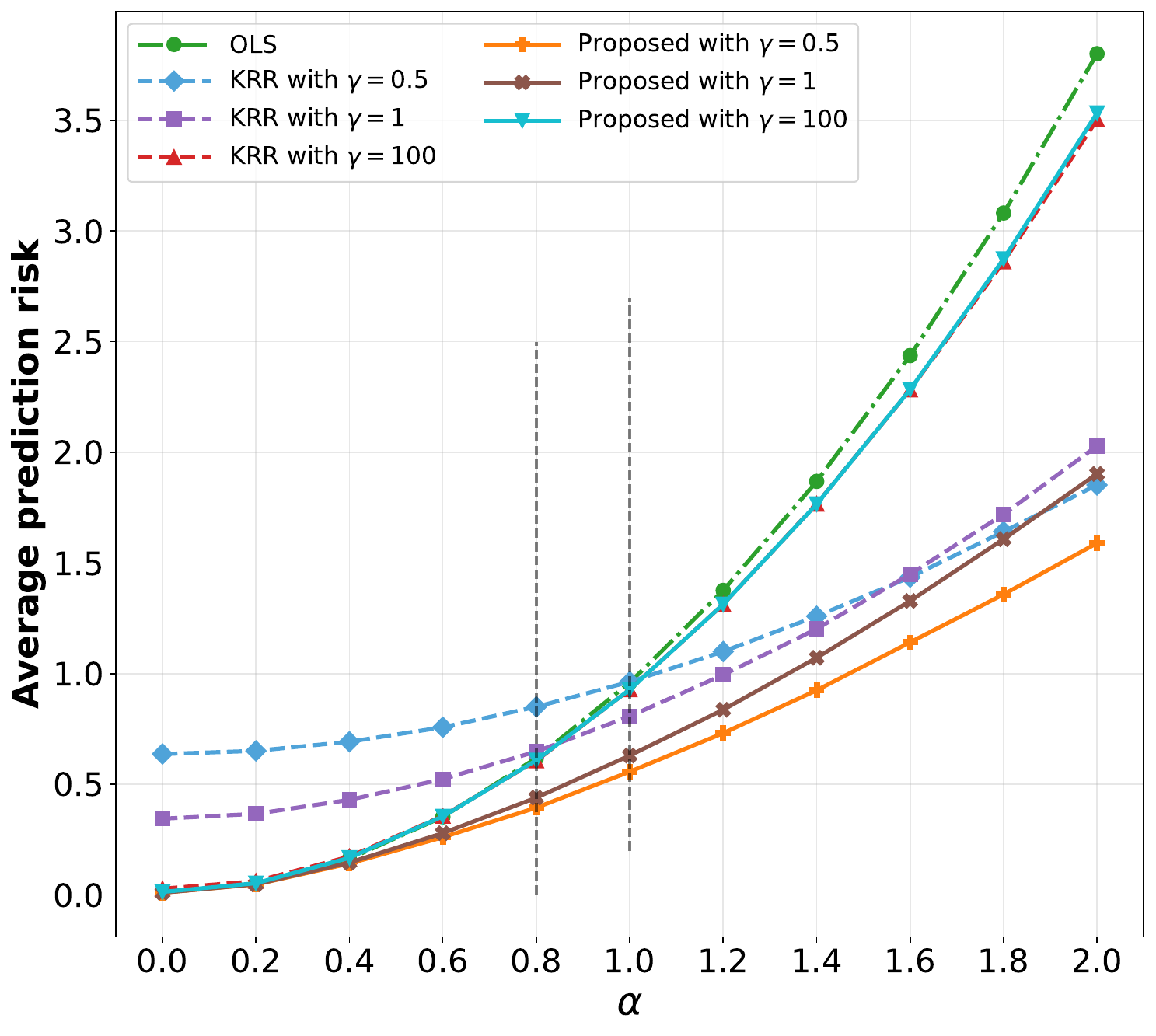}}
  \hspace{5mm}
        {\includegraphics[width=0.395\textwidth]{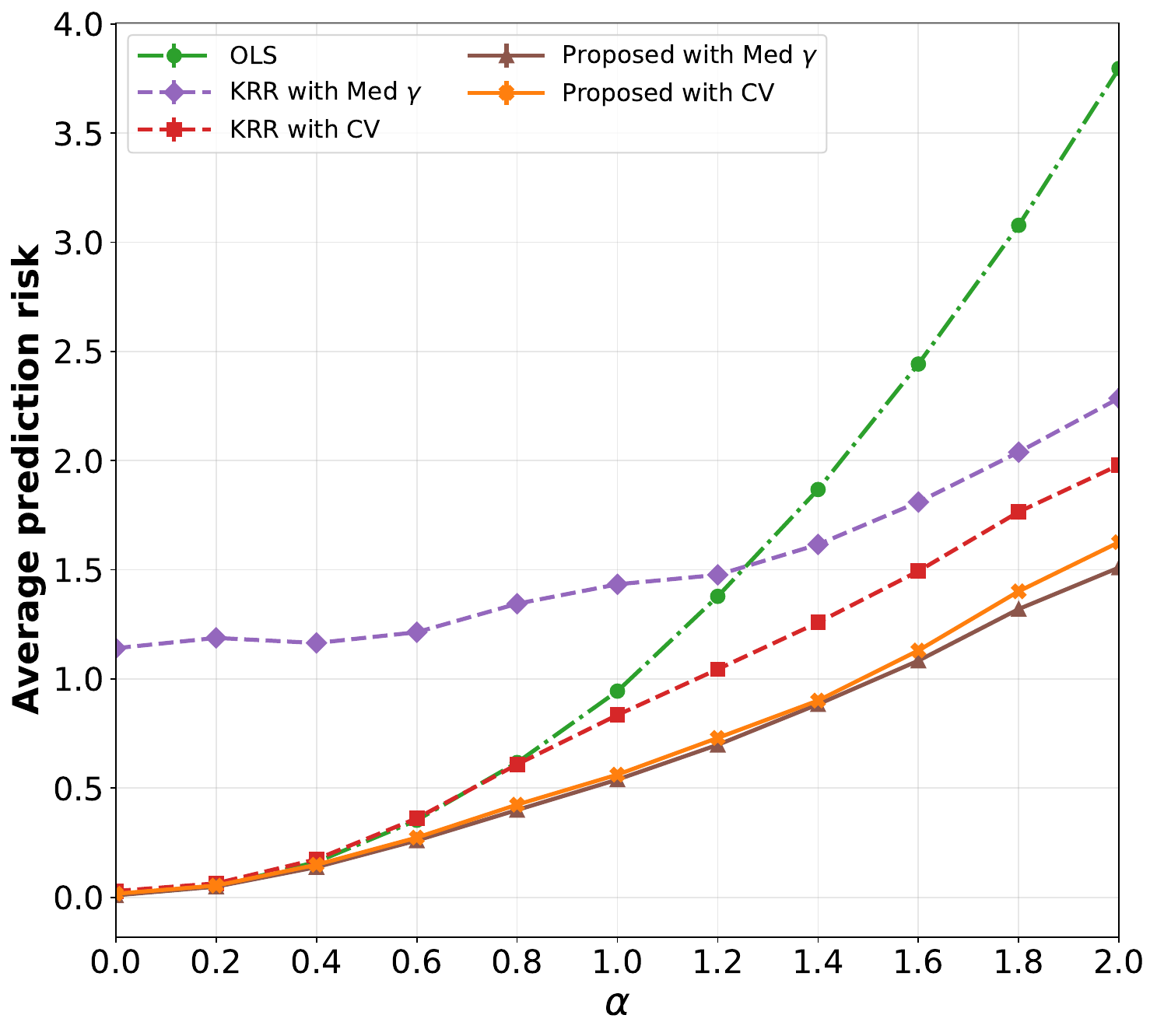}}
\caption{
The average prediction risk against $\alpha$. Left: the Gaussian kernel with 
fixed $\gamma$ in $\{0.5, 1,100\}$. Right: the Gaussian kernel with CV-selected and median-adjusted $\gamma$. 
}
\label{fig_vary_alpha_gaus}
\end{figure}

{
\setlength{\bibsep}{0.85pt}{
    \bibliographystyle{plainnat} 
    \bibliography{ref}  

@article{barron1999risk,
  title={Risk bounds for model selection via penalization},
  author={Barron, Andrew and Birg{\'e}, Lucien and Massart, Pascal},
  journal={Probability Theory and Related Fields},
  volume={113},
  number={3},
  pages={301--413},
  year={1999},
  publisher={Springer}
}

@article{birge2001gaussian,
  title={Gaussian model selection},
  author={Birg{\'e}, Lucien and Massart, Pascal},
  journal={Journal of the European Mathematical Society},
  volume={3},
  number={3},
  pages={203--268},
  year={2001}
}

@article{yang2011parametricness,
  title={Parametric or nonparametric? A parametricness index for model selection},
  author={Liu, Weidong and Yang, Yuhong},
  journal={The Annals of Statistics},
  volume={39},
  number={4},
  pages={2074--2102},
  year={2011},
  publisher={Institute of Mathematical Statistics}
}

@book{Wahba1990,
  author    = {Wahba, Grace},
  title     = {Spline Models for Observational Data},
  publisher = {Society for Industrial and Applied Mathematics},
  address   = {Philadelphia, PA},
  year      = {1990}
}

@article{Aronszajn1950,
  author  = {Nachman Aronszajn},
  title   = {Theory of Reproducing Kernels},
  journal = {Transactions of the American Mathematical Society},
  volume  = {68},
  number  = {3},
  pages   = {337--404},
  year    = {1950},
  doi     = {10.1090/S0002-9947-1950-0051437-7}
}

@article{mourtada2022elementary,
  title={An elementary analysis of ridge regression with random design},
  author={Mourtada, Jaouad and Rosasco, Lorenzo},
  journal={Comptes Rendus. Math{\'e}matique},
  volume={360},
  number={G9},
  pages={1055--1063},
  year={2022}
}

@book{GreenSilverman1994,
  author    = {Green, P. J. and Silverman, B. W.},
  title     = {Nonparametric Regression and Generalized Linear Models: A Roughness Penalty Approach},
  publisher = {Chapman and Hall},
  address   = {London},
  year      = {1994}
}

@book{Wood2017,
  author    = {Wood, Simon N.},
  title     = {Generalized Additive Models: An Introduction with R},
  edition   = {2},
  publisher = {CRC Press},
  address   = {Boca Raton, FL},
  year      = {2017}
}

@inproceedings{zhu2019high,
  title={High dimensional inference in partially linear models},
  author={Zhu, Ying and Yu, Zhuqing and Cheng, Guang},
  booktitle={The 22nd International Conference on Artificial Intelligence and Statistics},
  pages={2760--2769},
  year={2019},
  organization={PMLR}
}

@article{lee2024high,
  title={High-dimensional partial linear model with trend filtering},
  author={Lee, Sang Kyu and Hong, Hyokyoung G and Weng, Haolei and Loftfield, Erikka},
  journal={arXiv preprint arXiv:2410.20319},
  year={2024}
}

@article{yu2019minimax,
  title={Minimax optimal estimation in partially linear additive models under high dimension},
  author={Yu, Zhuqing and Levine, Michael and Cheng, Guang},
  year={2019}
}

@book{HardleLiangGao2000,
  author    = {{H{\"a}rdle}, Wolfgang and Liang, Hua and Gao, Jiti},
  title     = {Partially Linear Models},
  publisher = {Springer},
  address   = {Heidelberg},
  year      = {2000}
}

@article{zhu2017nonasymptotic,
  title={Nonasymptotic analysis of semiparametric regression models with high-dimensional parametric coefficients},
  author={Zhu, Ying},
  journal={The Annals of Statistics},
  pages={2274--2298},
  year={2017},
  publisher={JSTOR}
}

@article{muller2015partial,
  title={The partial linear model in high dimensions},
  author={M{\"u}ller, Patric and Van de Geer, Sara},
  journal={Scandinavian Journal of Statistics},
  volume={42},
  number={2},
  pages={580--608},
  year={2015},
  publisher={Wiley Online Library}
}

@article{smale2007learning,
  title={Learning theory estimates via integral operators and their approximations},
  author={Smale, Steve and Zhou, Ding-Xuan},
  journal={Constructive Approximation},
  volume={26},
  number={2},
  pages={153--172},
  year={2007},
  publisher={Springer}
}

@article{steinwart2012mercer,
  title={Mercer’s theorem on general domains: On the interaction between measures, kernels, and RKHSs},
  author={Steinwart, Ingo and Scovel, Clint},
  journal={Constructive Approximation},
  volume={35},
  number={3},
  pages={363--417},
  year={2012},
  publisher={Springer}
}

@article{bing2025kernel,
  title={Kernel ridge regression with predicted feature inputs and applications to factor-based nonparametric regression},
  author={Bing, Xin and He, Xin and Wang, Chao},
  journal={arXiv preprint arXiv:2505.20022},
  year={2025}
}

@inproceedings{steinwart2009optimal,
  title={Optimal rates for regularized least squares regression},
  author={Steinwart, Ingo and Scovel, Clint and others},
  booktitle={Proceedings of the Annual Conference on Learning Theory, 2009},
  pages={79--93},
  year={2009}
}

@article{caponnetto2007optimal,
  title={Optimal rates for the regularized least-squares algorithm},
  author={Caponnetto, Andrea and De Vito, Ernesto},
  journal={Foundations of Computational Mathematics},
  volume={7},
  pages={331--368},
  year={2007},
  publisher={Springer}
}

@article{mercer1909xvi,
  title={Functions of positive and negative type, and their connection the theory of integral equations},
  author={Mercer, James},
  journal={Philosophical Transactions of the Royal Society of London. Series A, Containing Papers of a Mathematical or Physical Character},
  volume={209},
  number={441-458},
  pages={415--446},
  year={1909},
  publisher={The Royal Society London}
}

@article{hsu2014random,
  title={Random design analysis of ridge regression},
  author={Hsu, Daniel and Kakade, Sham  and Zhang, Tong},
  journal={Foundations of Computational Mathematics},
  volume={14},
  pages={569--600},
  year={2014},
  publisher={Springer}
}

@article{micchelli2006universal,
  title={Universal Kernels.},
  author={Micchelli, Charles A and Xu, Yuesheng and Zhang, Haizhang},
  journal={The Journal of Machine Learning Research},
  volume={7},
  number={12},
  year={2006}
}

@article{zou2006adaptive,
  title={The adaptive lasso and its oracle properties},
  author={Zou, Hui},
  journal={Journal of the American statistical association},
  volume={101},
  number={476},
  pages={1418--1429},
  year={2006},
  publisher={Taylor \& Francis}
}

@article{alaoui2015fast,
  title={Fast randomized kernel ridge regression with statistical guarantees},
  author={Alaoui, Ahmed and Mahoney, Michael W},
  journal={Advances in neural information processing systems},
  volume={28},
  year={2015}
}

@article{tibshirani1996regression,
  title={Regression shrinkage and selection via the lasso},
  author={Tibshirani, Robert},
  journal={Journal of the Royal Statistical Society Series B: Statistical Methodology},
  volume={58},
  number={1},
  pages={267--288},
  year={1996},
  publisher={Oxford University Press}
}

@article{radhakrishnan2024mechanism,
  title={Mechanism for feature learning in neural networks and backpropagation-free machine learning models},
  author={Radhakrishnan, Adityanarayanan and Beaglehole, Daniel and Pandit, Parthe and Belkin, Mikhail},
  journal={Science},
  volume={383},
  number={6690},
  pages={1461--1467},
  year={2024},
  publisher={American Association for the Advancement of Science}
}

@article{singh2024kernel,
  title={Kernel methods for causal functions: dose, heterogeneous and incremental response curves},
  author={Singh, Rahul and Xu, Liyuan and Gretton, Arthur},
  journal={Biometrika},
  volume={111},
  number={2},
  pages={497--516},
  year={2024},
  publisher={Oxford University Press}
}

@inproceedings{an2007face,
  title={Face recognition using kernel ridge regression},
  author={An, Senjian and Liu, Wanquan and Venkatesh, Svetha},
  booktitle={2007 IEEE conference on computer vision and pattern recognition},
  pages={1--7},
  year={2007},
  organization={IEEE}
}

@article{duan2024optimal,
  title={Optimal policy evaluation using kernel-based temporal difference methods},
  author={Duan, Yaqi and Wang, Mengdi and Wainwright, Martin J},
  journal={The Annals of Statistics},
  volume={52},
  number={5},
  pages={1927--1952},
  year={2024},
  publisher={Institute of Mathematical Statistics}
}

@article{bartlett2005local,
  title={Local Rademacher Complexities},
  author={Bartlett, Peter L and Bousquet, Olivier and Mendelson, Shahar},
  journal={The Annals of Statistics},
  volume={33},
  number={4},
  pages={1497--1537},
  year={2005}
}

@article{fischer2020sobolev,
  title={Sobolev norm learning rates for regularized least-squares algorithms},
  author={Fischer, Simon and Steinwart, Ingo},
  journal={The Journal of Machine Learning Research},
  volume={21},
  number={1},
  pages={8464--8501},
  year={2020},
  publisher={JMLRORG}
}

@article{yang2017randomized,
  title={Randomized Sketches For Kernels: Fast And Optimal Nonparametric Regression},
  author={Yang, Yun and Pilanci, Mert and Wainwright, Martin },
  journal={The Annals of Statistics},
  volume={45},
  number={3},
  pages={991--1023},
  year={2017}
}

@book{edmunds1996function,
  title={Function Spaces, Entropy numbers, Differential operators},
  author={Edmunds, David Eric and Triebel, Hans},
  year={1996},
  publisher={Cambridge University Press}
}

@article{mukherjee2010learning,
  title={Learning gradients on manifolds},
  author={Mukherjee, Sayan and Wu, Qiang and Zhou, Ding-Xuan},
  journal={Bernoulli},
  volume={16},
  number={1},
  pages={181--207},
  year={2010}
}

@inproceedings{wahba1990spline,
  title={Spline models for observational data},
  author={Wahba, Grace},
  booktitle={CBMS-NSF Regional Conference Series in Applied Mathematics},
  year={1990},
  publisher={SIAM}
}

@article{kimeldorf1971some,
  title={Some results on {T}chebycheffian spline functions},
  author={Kimeldorf, George and Wahba, Grace},
  journal={Journal of Mathematical Analysis and Applications},
  volume={33},
  number={1},
  pages={82--95},
  year={1971},
  publisher={Elsevier}
}

@book{tsybakov2009introduction,
  title={Introduction to Nonparametric Estimation},
  author={Tsybakov, Alexandre B.},
  year={2009},
  publisher={Springer}
}

@article{rahimi2007random,
  title={Random features for large-scale kernel machines},
  author={Rahimi, Ali and Recht, Benjamin},
  journal={Advances in Neural Information Processing Systems},
  volume={20},
  year={2007}
}

@article{rudi2017generalization,
  title={Generalization properties of learning with random features},
  author={Rudi, Alessandro and Rosasco, Lorenzo},
  journal={Advances in Neural Information Processing Systems},
  volume={30},
  year={2017}
}

@book{steinwart2008support,
  title={Support {V}ector {M}achines},
  author={Steinwart, Ingo and Christmann, Andreas},
  year={2008},
  publisher={Springer Science \& Business Media}
}
}
}

\appendix

\newpage

\section*{Appendix}
 
The proofs are collected in \cref{app_sec_proof_LKRR,app_pf_ridge,app_comp_two}. \cref{app_comp_two} focuses on the comparison between the proposed procedure and the naive two-step procedure.   
\cref{app_sec_sim_ridge} contains additional simulation results of the proposed procedure in \cref{sec_ridge} using the ridge penalty.\\

\noindent{\it Notation.} For any $a,b\in \RR$, we write $a \wedge b= \min\{a,b\}$. Write $\cE =(\epsilon_1,\cdots,\epsilon_n)^\T\in \RR^n $.
Recall that for any function $f: \cX \to \RR$, write $f(\XX) =( f(X_1),\cdots,f(X_n))^\T \in \RR^n $.
 We use $I$  to represent the identity operator, with its domain depending on the context.  For any two positive semi-definite matrices $\AA,  \BB \in \RR^{n\times n}$, we write $\AA  \preceq  \BB$ 
if $v ^\T (\AA -\BB)v\le 0$ for all $v\in \RR^n$.  For any $\lambda \ge 0$, we write $\AA_\lambda := \AA + \lambda \II_n$.
For any matrix $\AA$, we use $\|\AA\|_\op$ to denote its operator norm.  Notation $\tr(\cdot)$ is used to denote the trace of any trace operator. 

\section{Main proofs for \cref{sec_theory_proposed}}\label{app_sec_proof_LKRR}

\subsection{Proof of the closed-form expression in \texorpdfstring{\cref{coeff}}{\texttwoinferior}}\label{app_pf_sec_2}
\begin{proof}[Proof of  \eqref{coeff}] 
Pick any $\lambda >0$.
Recall that
    \[
        \wh \omega =  \argmin_{\omega \in \RR^n }\l\{
            {1\over n}\l \|    Q_{\XX} (\YY - \sqrt{n} ~ \KK w ) \r \|^2 + \lambda \omega^\T \KK \omega
        \r\}.
    \]
    We claim that 
    $Q_{\XX} \KK+ \lambda \II_n$ is invertible. If this is the case, then  
since $\wh\omega$ satisfies 
    \[
        (\KK Q_{\XX} \KK + \lambda \KK) \wh\omega = {1\over \sqrt{n}} \KK Q_{\XX} \YY ,
    \]
 one solution of  $\wh\omega$ is 
\[
\wh \omega =  \l( Q_{\XX} \KK+ \lambda \II_n\r)^{-1} Q_{\XX} \YY.
\]
It remains to prove that  $Q_{\XX} \KK+ \lambda \II_n$ is invertible. To this end, suppose that there exists a $v\in \RR^n$ such that
$
( Q_{\XX} \KK + \lambda \II_n ) v  = 0 .
$
Then we have 
\[
\lambda Q_{\XX} v ~ = ~ -   Q_{\XX}^2 \KK  v ~ = ~  -   Q_{\XX} \KK  v ~ = ~ \lambda v 
\]
which, by $\lambda >0$, implies $Q_{\XX} v = v$. As a result, we have 
\[
 0= v^\T  Q_{\XX} \KK v+ \lambda v^\T v   
=  v^\T (\KK + \lambda \II_n ) v.
\]
Since $\KK$ is positive semi-definite under \cref{ass_pd_K}, we conclude  $v=0$ hence that  $Q_{\XX} \KK+ \lambda \II_n$ is  invertible.  This completes the proof.
\end{proof}

\subsection{Proof of \texorpdfstring{\cref{thm_fix}}{\texttwoinferior}}\label{app_proof_thm_fix}
For notational simplicity, write the eigen-decomposition of $\VV$ as 
\begin{equation}\label{V_eigen}
    \VV := Q_{\XX}\KK Q_{\XX} = \sum_{j=1}^n   \wh \nu_j  \wh v_j  \wh v_j^\T
\end{equation} 
with $\wh v_1,\ldots,\wh v_n \in \RR^n$ being the eigenvectors. Recall that   $\VV_\lambda = \VV + \lambda \II_n.$ 
For any $\lambda >0$, write
\begin{align} \label{def_P_lambda}
    P_\lambda = \KK  ( Q_{\XX} \KK +\lambda \II_n )^{-1}  Q_{\XX}, \qquad Q_{\lambda} =\II_n -P_\lambda.
\end{align}
The following lemma gives useful expressions of $Q_{\XX} P_\lambda$ and $Q_{\XX} Q_{\lambda}$.  

\begin{lemma}\label{lem_rem_Q_PQ_lbd}
    Let  $P_\lambda$ and $Q_{\lambda}$ be defined in \eqref{def_P_lambda}. One has 
    \begin{equation}\label{eq_QX_Plbd}
        \begin{aligned}
             Q_{\XX} P_\lambda &= Q_{\XX} \KK^{1/2} (\KK^{1/2} Q_{\XX} \KK^{1/2} +\lambda \II_n )^{-1}\KK^{1/2}Q_{\XX}\\
            &= Q_{\XX} \KK Q_{\XX}  (Q_{\XX} \KK Q_{\XX}+\lambda \II_n )^{-1}\\
            &= \VV \VV_\lambda^{-1}
        \end{aligned}
    \end{equation}
    and 
    \begin{align} \label{eq_QX_Qlbd}
        Q_{\XX} Q_{\lambda} &= \lambda  Q_{\XX}   (  Q_{\XX}  \KK Q_{\XX}  +\lambda \II_n)^{-1}   Q_{\XX} = \lambda Q_{\XX}  \VV_\lambda^{-1} Q_{\XX}.
    \end{align}
    Immediately, we have $\| Q_{\XX} P_\lambda \|_\op\le 1 $ and $\| Q_{\XX} Q_{\lambda} \|_\op\le 1$.
\end{lemma} 
\begin{proof} 
The fact that both operator norms are bounded by one is trivial. 

To prove \eqref{eq_QX_Plbd}, applying the  Woodbury matrix identity gives
\[
( Q_{\XX} \KK +\lambda \II_n )^{-1}=\lambda^{-1}\II_n -\lambda^{-1} Q_{\XX}\KK^{1/2}(\KK^{1/2} Q_{\XX} \KK^{1/2} +\lambda \II_n )^{-1}\KK^{1/2},
\]
so that
\begin{align*}
   \KK ( Q_{\XX} \KK +\lambda \II_n )^{-1} ~=~&\lambda^{-1}\KK^{1/2}\l(\II_n - \KK^{1/2} Q_{\XX}\KK^{1/2}(\KK^{1/2} Q_{\XX}\KK^{1/2} +\lambda \II_n )^{-1}\r)\KK^{1/2}\\
  ~=~ &\KK^{1/2} (\KK^{1/2} Q_{\XX} \KK^{1/2} +\lambda \II_n )^{-1}\KK^{1/2}.
\end{align*}
By noting that
\[
  Q_{\XX} P_\lambda =   Q_{\XX} \KK ( Q_{\XX} \KK +\lambda \II_n )^{-1} Q_{\XX},
\]
we conclude the first equality in \eqref{eq_QX_Plbd}.
The second equality in \eqref{eq_QX_Plbd} follows by repeating the same arguments.

Equation \eqref{eq_QX_Qlbd} follows by noting that
    \begin{align*}
    Q_{\XX}   Q_{\lambda} 
    & = Q_{\XX}  ( \II_n -   \KK  (  Q_{\XX}  \KK +\lambda \II_n)^{-1}  Q_{\XX}  ) &&\text{by \eqref{def_P_lambda}}\\
  &  = Q_{\XX}  \b( \II_n - Q_{\XX}  \KK  (  Q_{\XX}  \KK +\lambda \II_n)^{-1}  Q_{\XX}  \b)Q_{\XX} &&\text{since $Q_{\XX} $ is projection}\\
    &  = Q_{\XX}  \b(\II_n - Q_{\XX}  \KK  Q_{\XX}   ( Q_{\XX}  \KK  Q_{\XX}   + \lambda \II_n)^{-1}  \b) Q_{\XX}  &&\text{by \eqref{eq_QX_Plbd}} \\
   & =\lambda ~ Q_{\XX}   (  Q_{\XX}  \KK Q_{\XX}  +\lambda \II_n)^{-1}   Q_{\XX}.
\end{align*}  
The proof is complete.
\end{proof}
 
\medskip

\begin{proof}[Proof of \cref{thm_fix}]
Using \eqref{fit} and \eqref{def_P_lambda} gives 
\begin{align} \label{eq_decom_f_hat} \nonumber
        \wh f(\XX) - f^*  (\XX )     
 & ~ =~      P_{\XX}\YY  + Q_{\XX}P_{\lambda} \YY - f^*  (\XX)\\  \nonumber
& ~ =~     P_ \XX  \cE  +  Q_{\XX}P_{\lambda}  \cE +   Q_{\XX} P_{\lambda}  f^*(\XX) - Q_{\XX}f^*(\XX) &&\text{by $\YY = f^*(\XX) + \cE$ }\\
&~ = ~ P_ \XX  \cE  +  Q_{\XX}P_{\lambda} \cE -  Q_{\XX} Q_{\lambda} f^*(\XX).
\end{align} 
It then follows from \eqref{eq_QX_Qlbd} that 
\begin{align}\label{eq_decom_1}
 \cR(\wh f) &= {1\over n} \EE\left\|
  P_ \XX  \cE     +  Q_{\XX} P_\lambda  \cE 
 \right\|^2 + {1\over n} \left\|
\lambda Q_{\XX}  \VV_\lambda^{-1} Q_{\XX} f^*(\XX)
 \right\|^2 =: \rI + \rII.  
\end{align} 
Using  $\rank(\XX) = d$ gives 
\begin{align*}
    \rI  &=    {\sigma^2 \over n}  \tr\l( 
  (P_ \XX     +  Q_{\XX} P_\lambda)^\T  (P_ \XX     +  Q_{\XX} P_\lambda)\r)\\ & =   {\sigma^2  \over n} \left\{ d +    \tr\l(P_\lambda^\T Q_{\XX}  P_\lambda\r)\right\}\\
  &=  {\sigma^2  \over n} \left\{ d +   \tr\l(\VV^2(\VV+\lambda \II_n)^{-2}\r)\right\} &&\text{by \eqref{eq_QX_Plbd}}.
\end{align*}  
The eigen-decomposition of $\VV$ gives the variance bound.

To bound $\rII$,  let
    \begin{equation}\label{def_f_lambda}
        f_\lambda =  \argmin_{f\in \cH_K}\l\{
            {1\over n}\|Q_{\XX}  f(\XX) - Q_{\XX}  f^*(\XX)  \|^2 + \lambda \|f\|_K^2
        \r\}.
    \end{equation}
    We claim that  
  \begin{equation}\label{target_Qf}
        Q_{\XX} f_\lambda(\XX) = Q_{\XX}  P_\lambda  f^*(\XX)  =  Q_{\XX} \VV\VV_\lambda^{-1}  Q_{\XX} f^*(\XX).
    \end{equation}
    If this is the case, then  
    \begin{align*}
 \rII=        {1\over n}\| \lambda Q_{\XX}  \VV_{\lambda}^{-1} Q_{\XX} f^*  (\XX)\|^2  &=  {1\over n}\|Q_{\XX} f_\lambda(\XX) - Q_{\XX}  f^*(\XX)\|^2 \\
        &\le   \inf_{f\in \cH_K}\l\{
            {1\over n}\|Q_{\XX} f(\XX) - Q_{\XX}  f^*(\XX)\|^2 + \lambda \|f\|_K^2
        \r\}\\ 
        &= \inf_{\alpha\in \RR^d, f\in \cH_K}\l\{
            {1\over n}\| f(\XX) +\XX\alpha -  f^*(\XX) \|^2 + \lambda \|f\|_K^2
        \r\}
    \end{align*} 
    yielding the desired claim. 
    
    It remains to show \eqref{target_Qf}. 
The first equality of \eqref{target_Qf} follows from the same argument of  deducing the fit $\wh f(\XX)$ in \cref{sec_method_krr},   with $f^*(\XX)$ in place of $\YY$.
    The second equality follows by 
    \[
     Q_{\XX}  P_\lambda  f^*(\XX) \overset{\eqref{def_P_lambda}}{=} Q_{\XX}  (Q_{\XX}  P_\lambda)  Q_{\XX} f^*(\XX) \overset{\eqref{eq_QX_Plbd}}{=} Q_{\XX} \VV\VV_\lambda^{-1}  Q_{\XX} f^*(\XX).
    \]
This completes the proof. 
\end{proof}

\subsection{Proof of \texorpdfstring{\cref{thm_krr_fix}}{\texttwoinferior}}\label{app_proof_thm_krr_fix}
\begin{proof}
    Pick any $\lambda >0$ and recall that $\KK_{\lambda} := \KK + \lambda  \II_n$.
        Observe that
        \begin{align*}
            f^*  (\XX) -    \wt g (\XX) 
           ~ =~  
        f^*  (\XX)  -   \KK \KK_{\lambda}^{-1} (f^*  (\XX) + \cE  ) 
           ~ =~     \lambda  \KK_{\lambda}^{-1}  f^*  (\XX) - \KK\KK_{\lambda}^{-1} \cE.
            \end{align*}
     By the  same arguments of proving \cref{thm_fix}, one has 
    \begin{align*}
        \cR(\wt g) = {1\over n}\|\lambda  \KK_{\lambda}^{-1}  f^*  (\XX)\|^2 + {\sigma^2 \over n} \tr\l(\KK^2\KK_{\lambda}^{-2}\r).
    \end{align*} 
    Using the eigen-decomposition of $\KK$ gives the variance bound. 
    To bound the first term, let
    \[
        f_\lambda =  \argmin_{f\in \cH_K}\l\{
            {1\over n}\|f(\XX) - f^*(\XX)\|^2 + \lambda \|f\|_K^2
        \r\}
    \]
    so that, by the representer theorem, $f_\lambda(\XX) = \KK \KK_{\lambda}^{-1} f^*(\XX)$, hence 
    \begin{align*}
        {1\over n}\|\lambda  \KK_{\lambda}^{-1}  f^*  (\XX)\|^2  =  {1\over n}\|f_\lambda(\XX) - f^*(\XX)\|^2 \le   \inf_{f\in \cH_K}\l\{
            {1\over n}\|f(\XX) - f^*(\XX)\|^2 + \lambda \|f\|_K^2
        \r\}.
    \end{align*}
     The proof is complete.
\end{proof}

\subsection{Proof of \cref{var_comp}}\label{app_proof_lemma_thm_fix}

The following lemma characterizes the relationship between the eigenvalues of $\KK$ and $Q_{\XX} \KK Q_{\XX}$.   

\begin{lemma}\label{lem_comp}
For any $\lambda>0$, one has 
\[
  \sum_{i=1}^n  \left(\wh\nu_i \over \wh\nu_i + \lambda\right)^2~ \le  ~  \sum_{i=1}^n  \left(\wh\mu_i \over \wh\mu_i + \lambda \right)^2 
\]
\end{lemma}
\begin{proof}
Start by noting that $\KK^{1/2} Q_{\XX} \KK^{1/2} $  has the same eigenvalues as $ Q_{\XX} \KK Q_{\XX} $. 
Since 
$\KK^{1/2} Q_{\XX} \KK^{1/2} ~\preceq~  \KK $, applying the 
minimax theorem 
gives that for each $i\in [n]$, 
\begin{align*}
    \wh\nu_i   ~=~ \min_{\dim(\SS) = n-i+1}~  \max _{x\in \SS,\|x\|=1} \l\langle x , \KK^{1/2} Q_{\XX} \KK^{1/2}   x\r \rangle  ~  \le ~
\min_{\dim(\SS) = n-i+1}~  \max _{x\in S,\|x\|=1} \l\langle x , \KK    x\r \rangle =\wh\mu_i .  
\end{align*}
The proof is completed by noting that $f(t) = t^2 / (t+\lambda)^2$ is non-decreasing on $[0,\infty)$. 
\end{proof}

\subsection{Proof of \texorpdfstring{\cref{thm_random}}{\texttwoinferior}}\label{app_proof_thm_random}

\begin{proof} 
    In view of \cref{thm_fix} and \eqref{V_eigen}, it remains to bound from above  
    \[
        {1\over n}  \EE \tr\l(\VV^2 \VV_\lambda^{-2}\r).
    \]
  To this end,    by  \cref{lem_comp}, we observe that  
    \begin{equation}\label{eq_tr_bd} 
 \tr\l(\VV^2 \VV_\lambda^{-2}\r)
    ~\le  ~    \sum_{k=1}^n \frac{\wh \mu_k^2 }{(\wh \mu_k+\lambda)^2} 
   ~\le~  \sum_{k=1}^n \frac{\wh \mu_k}{\wh \mu_k+\lambda}~ =~\tr( \KK \KK_{\lambda}^{-1}). 
        \end{equation}
Define the empirical covariance operator $\wh T_K:\cH_K \to\cH_K$  as
\[
    \wh T_K f  = \frac{1}{n}\sum_{i=1}^n K(X_i, \cdot) f(X_i) ,\qquad \forall ~f\in \cH_K. 
\]
Referring to \cite{bartlett2005local} and Lemma 8 in \cite{bing2025kernel}\footnote{The number of non-zero eigenvalues of $\wh T_K$ is at most $n$. Moreover, the eigenvalues of  $\KK$ are the same as the $n$ largest eigenvalues of $\wh T_K$.}, we have 
\[
\tr( \KK \KK_{\lambda}^{-1}) ~ =~   \tr\l( \EE[\wh T_K]   (\EE [\wh T_K]  +\lambda I)^{-1} \r) .
\]
Note that the population-level counterpart of $\wh T_K$ is $T_K:\cH_K \to\cH_K $, given by 
$$
    T_K f  = \int_{\cX} f(x)K(x,\cdot) \d \rho (x),\qquad  \forall ~f\in \cH_K. 
$$
We have $\EE[ \wh T_K  ] = T_K$.
Moreover, the covariance operator $T_K$
 shares the same eigenvalues as the integral operator $L_K$ in \eqref{def_L_K} \citep{caponnetto2007optimal}.
Therefore, together with \eqref{eq_tr_bd} and
applying Jensen's inequality, by noting that the function $t\mapsto t/(t+\lambda)$ is concave on $[0,\infty)$,
gives
\begin{align*}
    \EE\l[\tr\l( \VV^2 \VV_\lambda^{-2}\r) \r]& ~\le ~  \tr\l( \EE\l[ \wh T_K  (\wh T_K  +\lambda I)^{-1}\r] \r)\\
    & ~\le ~        \tr\l( \EE[\wh T_K]   (\EE [\wh T_K]  +\lambda I)^{-1} \r)  \\
    &  ~= ~      \tr\l(  T_K   (T_K +\lambda I)^{-1} \r)   ~= ~   \sum_{k=1}^\i \frac{\mu_k}{\mu_k+\lambda}. 
\end{align*}
For any positive values $a$ and $b$, we have
\[
    \frac{ab}{a+b}  \le   \frac{1}{\max\{1/a,1/b\}} = a \wedge b,
\]
from which we  further have
\begin{align} \label{eq_kernel_comple}
  \EE\l[\tr\l( \VV^2  \VV_\lambda^{-2}\r) \r] 
      ~  \le~  \lambda^{-1} \sum_{k=1}^\i \min\{\lambda, \mu_k\} ~  \overset{\eqref{kernel_complexity}}{=}~  n \lambda ^{-1} R(\lambda)^2.
\end{align} 
The proof is complete. 
\end{proof}

The following lemma bounds from below $\delta_n$ when $ 1/n \lesssim  \mu_1$.

\begin{lemma}\label{lem_lb_delta_n}
Assuming  that $\mu_1 \ge c/n$ for some positive constant $c$, the fixed point $\delta_n$ of $R(\cdot)$ satisfies
\[
    \delta_n   \ge \frac{c'}{n} , 
    \]
    where $c'$ is some constant. 
\end{lemma}
\begin{proof}
If $\delta_n > c/n$, we complete the proof. It remains to consider the case $\delta_n \le c/n$. 
By the definition of 
$R(\cdot)$ in 
\cref{kernel_complexity}, we note that
\begin{align}\label{eq_c_n}
        R\l(\frac{c}{n}\r)^2  =   \frac{1}{n} \sum_{j=1}^\infty 
        \min \l\{\mu_j,\frac{c}{n}\r\}  \ge   \frac{\mu_1}{n}  \wedge  \frac{c}{n^2}  \ge \frac{c}{n^2}.
    \end{align}
Since the function $R(\delta)/\sqrt{\delta}$ is non-increasing in $\delta$, we have 
    \[
    \sqrt{\delta_n}  = \frac{R(\delta_n)}{ \sqrt{\delta_n}} \ge  \frac{R(c/n)}{ \sqrt{c/n}} \overset{\eqref{eq_c_n}}{\ge}\frac{1}{\sqrt{n}}, 
    \]
    so that
    \[
   \delta_n\ge 1/n. 
    \]
    Combining these two cases gives the desired result.    
\end{proof}

Recall that
a function $\psi: [0,\i)\to  [0,\i)$ is a sub-root function if it is nonnegative, non-decreasing, and if $\delta\mapsto \psi(\delta)/\sqrt{\delta}$ is non-increasing for $\delta>0$.
A useful property of the sub-root function  
is stated as follows. 

\begin{lemma}[Lemma 3.2 in \cite{bartlett2005local}]\label{lem_subroot_monotone}
If $\psi:[0, \infty) \to [0, \infty)$ is a nontrivial\footnote{Functions other than the constant function $\psi\equiv 0$. } 
sub-root function, 
then it is continuous on $[0, \infty)$ and the equation $\psi(\delta)=\delta$ has a unique positive solution $\delta_\star$. Moreover, for all $\delta>0, \delta\geq \psi(\delta)$ if and only if $\delta_\star \leq \delta$. Here, $\delta_\star$ is referred to as the fixed point. 
\end{lemma}

 \section{Comparison between the proposed method and two-step algorithm}\label{app_comp_two}

In this section, we start to compare the proposed predictor $\wh f$ with the naive two-step predictor $\wh f_{\two}$ detailed in \cref{rem_twostep} from the theoretical perspective, while their numerical comparison is presented in \cref{sec_simu_comp_two}. 
Recall that for any $\lambda>0$, $\KK_{\lambda} =\KK +\lambda \II_n$.  It is easy to see that the solution $\wh g_{\two}$ satisfies
\begin{equation}\label{def_g_two} 
        \wh g_{\two} =\frac{1}{\sqrt{n}}\sum_{i=1}^n \wt w_i K(X_i,\cdot), \qquad \text{with} \qquad  \wt w=  {1\over \sqrt n}\KK_{\lambda}^{-1}Q_{\XX} \YY. 
\end{equation}
The following theorem states the prediction risk bound of 
$\wh f_{\two}$. Its proof is given in the end of this section.

\begin{theorem}\label{thm_twostep}
    Grant model (\ref{model}) and
    Assumption   \ref{ass_pd_K}.  For any $\lambda >0$, one has 
    \begin{equation}
        \label{risk_bd_two}
        \begin{aligned}
      \cR(\wh f_{\rm two}-f^*) ~ \le ~    & \inf_{\alpha\in \RR^d, f\in \cH_K}\left\{  \cR(\langle \alpha,\rangle + f- f^*)  +  {1\over n} \|P_{\XX} f(\XX)\|^2+  \lambda \|f\|_K^2  \right\}\\
      & ~~ + {\sigma^2  \over n}\left\{ d +  \tr\l(Q_{\XX} \KK^2 \KK_{\lambda}^{-2} \r)\right\}
\end{aligned}
    \end{equation}
\end{theorem}

We proceed to compare the proposed $\wh f$ with $\wh f_{\two}$ through their prediction risk bounds, with the risk bound for the former is presented in \cref{thm_fix}. For the bias and approximation error, 
we find that the irreducible term in \eqref{risk_bd_two} is the same as that for  $\wh f$, and   
its leading constant is also equal to one. 
However, the two-step method introduces an additional bias $(1/n) \|P_{\XX} f(\XX)\|^2$, which can be large, especially when the linear space spanned by the columns of $\XX$ and the RKHS $\cH_K$ are highly aligned. Comparing the variances, as shown in \cref{lem_var_com}, the variance of $\wh f$ can be bounded from above by that of $\wh f_{\two}$ for any given $\lambda$, up to a parametric order $ {\sigma^2 d / n}$. The order $ {\sigma^2 d / n}$ can be absorbed into the variance of OLS and, as discussed in \cref{rem_var}, is negligible under general settings. Summarizing, our jointly optimizing procedure has a clear advantage over the naive 
two-step algorithm. \\

Recall $\VV$ from  \eqref{V_eigen} and $\VV_\lambda = \VV + \lambda \II_n$. 
The following lemma states that for any given $\lambda>0$, 
the variance  $ \tr(\VV^2\VV^{-2}_\lambda) $ in \cref{thm_fix} is not greater than the variance $\tr(Q_{\XX} \KK^2 \KK_{\lambda}^{-2})$ in \cref{thm_twostep}, up to a parametric order $ {\sigma^2 d / n}$. 
\begin{lemma}\label{lem_var_com}
For any $\lambda>0$, one has 
    \[
\tr\l( Q_{\XX} \KK^2 \KK_{\lambda}^{-2} \r)  ~\le ~    \tr(\VV^2\VV^{-2}_\lambda)  + \sigma^2  { d \over n}. 
    \]
\end{lemma}
\begin{proof}
Note that  for any $\lambda>0$, 
\begin{align*}
\tr\l(Q_{\XX} \KK^2 \KK_{\lambda}^{-2} \r)  ~  = ~  &   \tr\l( \KK^2 \KK_{\lambda}^{-2}  \r) -  \tr\l(   P_{\XX}  \KK^2 \KK_{\lambda}^{-2} \r)   \\
 ~ \ge ~   &   \tr\l( \VV^2 \VV_\lambda^{-2} \r)    -   \tr\l( P_{\XX}  \KK^2 \KK_{\lambda}^{-2} \r) &&\text{by  \cref{lem_comp},} 
\end{align*}
so that 
\begin{align*}
 \tr\l( \VV^2 \VV_\lambda^{-2} \r)    ~ \le ~&  \tr\l(Q_{\XX} \KK^2 \KK_{\lambda}^{-2} \r)   + {\sigma^2  \over n}   \tr\l( P_{\XX}  \KK^2 \KK_{\lambda}^{-2} \r)   \\
~ \le ~&    \tr\l(Q_{\XX} \KK^2 \KK_{\lambda}^{-2} \r)  +  {\sigma^2  \over n} ~ \|\KK\KK_{\lambda} ^{-1} \|_\op^2~  \tr(P_{\XX})  \\
~ \le ~&    \tr\l(Q_{\XX} \KK^2 \KK_{\lambda}^{-2} \r)  +\sigma^2   { d \over n} &&\text{by $\rank(\XX) =d$}.
\end{align*}
The proof is complete. 
\end{proof}

\begin{proof}[Proof of \cref{thm_twostep}]
  We follow similar arguments in the proof of \cref{thm_fix}.   Start with
  \begin{align*}
 \wh f_{\two} (\XX) - f^*  (\XX ) 
&
~ =~ P_ \XX  \YY  +  \KK\KK_{\lambda} ^{-1} Q_{\XX} \YY -  f^*  (\XX) 
\\
&
~ =~ P_ \XX  \cE   +   \KK\KK_{\lambda} ^{-1} Q_{\XX} \cE +  \KK\KK_{\lambda}^{-1}   Q_{\XX}  f^*  (\XX)    - Q_{\XX} f^*  (\XX) \\
&
~ =~ P_ \XX  \cE   +   \KK\KK_{\lambda}^{-1}  Q_{\XX} \cE -  \lambda  \KK_{\lambda}^{-1} Q_{\XX}  f^*  (\XX) 
\end{align*}
Therefore, the prediction risk of the predictor $\wh f_{\two}$  can be decomposed into   
\[
\cR(\wh f_{\two}) ~ =~ 
 {1 \over n} \EE \l\| P_ \XX  \cE   +   \KK\KK_{\lambda} ^{-1}  Q_{\XX} \cE\r\| ^2 +  {1 \over n}\l \|  \lambda \KK_{\lambda}^{-1}  Q_{\XX}  f^*  (\XX)\r\|^2 .
\]
Further note that
\begin{equation}
    \label{decom_two}
    \begin{aligned}
 {1 \over n} \EE \| P_ \XX   \cE   + \KK\KK_{\lambda}  Q_{\XX} \cE\| ^2  
  & =   {\sigma^2  \over n} \left\{ d +    \tr\l(\KK\KK_{\lambda} ^{-1}Q_{\XX}  \KK\KK_{\lambda} ^{-1}\r) +2 \tr\l(P_{\XX}\KK\KK_{\lambda} ^{-1}Q_{\XX}  \r)\right\}
   \\
   & ={\sigma^2  \over n} \left\{ d +     \tr\l( Q_{\XX} \KK^2 \KK_{\lambda}^{-2} \r)  \right\}
   ,
\end{aligned}
\end{equation}
where  the last step follows from 
\begin{align*}
    \tr\l(P_{\XX} \KK \KK_{\lambda} ^{-1}Q_{\XX}  \r)     & =    \tr\l(P_{\XX}\KK\KK_{\lambda} ^{-1} \r)- \tr\l(P_{\XX}\KK\KK_{\lambda} ^{-1}  P_{\XX} \r) \\
    & = \tr\l(P_{\XX}\KK\KK_{\lambda} ^{-1} \r)- \tr\l(P_{\XX}\KK\KK_{\lambda} ^{-1} \r) = 0. 
\end{align*}

To bound the second term in the right-hand of \eqref{decom_two}, we follow the same argument of proving \cref{thm_krr_fix} by replacing $f^*(\XX)$ by $Q_{\XX} f^*(\XX)$ to obtain
\begin{align*}
     {1 \over n}\l \|  \lambda \KK_{\lambda}^{-1}  Q_{\XX}  f^*  (\XX)\r\|^2 ~ \le~   \inf_{f\in \cH_K}\l\{
            {1\over n} \|f(\XX) -Q_{\XX} f^*(\XX)\|^2 + \lambda \|f\|_K^2
        \r\}.
\end{align*}
By noting that
\begin{align*}
   \|f(\XX) -Q_{\XX} f^*(\XX)\|^2   &  
   = \|P_{\XX} f(\XX) +  Q_{\XX} f(\XX) -Q_{\XX} f^*(\XX)\|^2 \\
   & = \|P_{\XX} f(\XX) \|^2+ \|  Q_{\XX} f(\XX) -Q_{\XX} f^*(\XX)\|^2
   \\
   & = \|P_{\XX} f(\XX) \|^2+ \inf_{\alpha\in \RR^d} ~\l \| \XX\alpha+  f(\XX) - f^*(\XX)\r\|^2, 
\end{align*}
we complete the proof. 
\end{proof}

\subsection{Numerical comparison between the proposed method and two-step algorithm}\label{sec_simu_comp_two}
In this section we conduct simulation studies to numerically verify the 
advantages of the adaptive KRR (AKRR) proposed  in \cref{sec_method_krr} over the 
two-step procedure (TKRR) described in \cref{rem_twostep}.

We consider model \eqref{model} with 
\begin{align} \label{regre_function}
    f^*(X)
=
\frac{1}{\sqrt d}
\sum_{j=1}^d (-1)^{j-1}X_j
+
\alpha
\l[
\frac{1}{\sqrt d}
\sum_{j=1}^d
\l\{
\sin(\pi X_j)
+
\frac12\cos(2\pi X_j)
\r\}
+
\frac{\sum_{j=1}^{d-1}
\sin(X_jX_{j+1})}{2\sqrt{d-1}}
\r],
\end{align}
 $\epsilon \sim N(0,1)$, and $X$ 
 follows the  multivariate normal distribution $N(0,  \Sigma )$, where
\begin{align} \label{X_covariance}
    \Sigma_{jk}
: = 
\varrho^{|j-k|/s},
\qquad 1\le j,k\le d.
\end{align}
In this experiment, 
we choose $\varrho =0.9$ and $s=6$,  the sample size is set to $n=400$, and the dimension of $X$ is set to $d=20$.  
We vary $\alpha$ to examine its effect on the prediction performance of  AKRR and TKRR.  
We adopt the Gaussian kernel $K(x,x') =\exp(-\|x-x'\|^2/\gamma)$ with $\gamma>0$. The regularization parameter $\lambda$ is always 
selected via cross-validation (CV). Following the same treatment in  \cref{sec_gaussian}, we consider both the median-adjusted and CV-selected $\gamma$. 
For the latter choice,   $\lambda$
and
$\gamma$ are jointly selected via CV.
The obtained results are reported in Figure \ref{akrr_vs_tkrr}. We observe that, regardless of how  $\gamma$ is selected, the proposed method consistently improves over TKRR. The improvement becomes more evident as the nonlinear strength in the regression function increases.

\begin{figure}[ht]
    \centering
        {\includegraphics[width=0.5\textwidth]{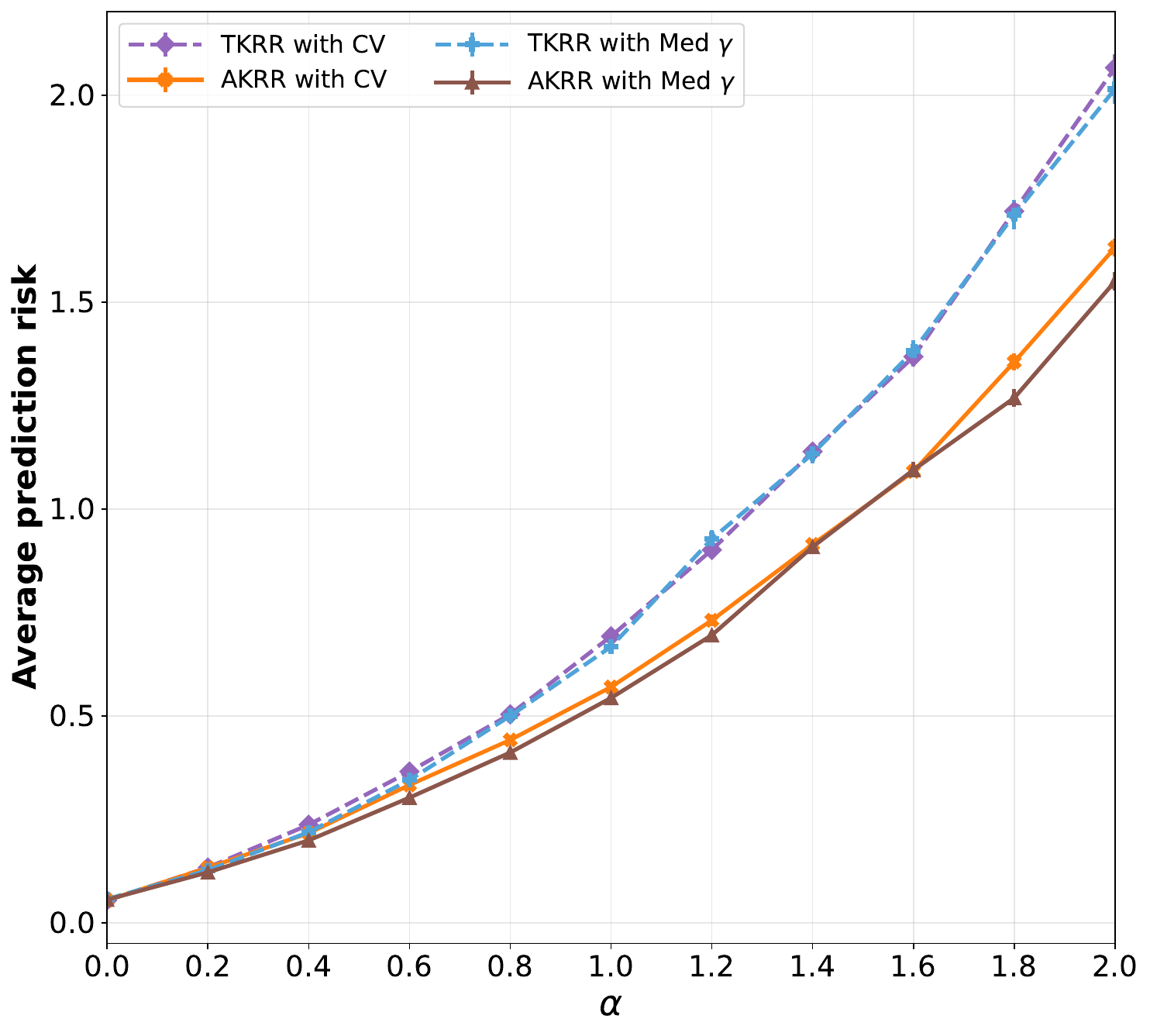}}
\caption{
The average prediction risk against $\alpha$.
}
\label{akrr_vs_tkrr}
\end{figure}


\section{Main Proofs for \texorpdfstring{\cref{sec_ridge}}{\texttwoinferior}} \label{app_pf_ridge}

\subsection{Proof of \texorpdfstring{\cref{closed_form_ridge}}{\texttwoinferior}}\label{app_pf_closed_form_ridge}

\begin{proof}[Proof of \eqref{closed_form_ridge}]
For any fixed $g \in \cH_K$, by solving the corresponding profile optimization problem
\[
\wh \alpha_{r,g}
~ =~  \argmin_{\alpha\in \RR^d}  \l\{ \frac{1}{n} \sum_{i=1}^n\l(Y_i -   X_i ^\T  \alpha - g(X_i)\r)^2 +  \mu \|\alpha\|^2
+\lambda\|g\|_{K}^2  \r\},
\]
we obtain 
\begin{align}\label{def_alpha_r_g}
    \wh \alpha_{r,g} = ( \XX ^\T \XX+ n \mu  ~ \II_d) ^{-1}  \XX^\T  (\YY - g(\XX)).
\end{align}
By substituting $\wh \alpha_{r,g}$ into 
\eqref{def_method_ridge}, we obtain
\[
\wh g_r
~ =~ \argmin_{ g \in \cH_K}  \left\{ \frac{1}{n} \sum_{i=1}^n\l(Y_i- X_i ^\T  \wh \alpha_{r,g}   - g(X_i)\r)^2 + \mu \|\wh \alpha_{r,g} \|^2  +\lambda \|g\|_K^2 \right\}.
\]
For any $\mu\ge 0$, recall 
\begin{align} \label{def_P_mu}
    P_{\mu}  = \XX ( \XX ^\T \XX+ n \mu  ~ \II_d) ^{-1}  \XX^\T,\qquad   Q_\mu =\II_n- P_\mu. 
\end{align}
Armed with \cref{lem_ridge},
in conjunction with the representer theorem \citep{kimeldorf1971some},  the optimal solution $\wh g_r$  has the form of 
\[
\wh g _r ={ 1\over \sqrt{n}} \sum_{i=1}^n \wh \beta_i K(X_i,\cdot),
\]
where the coefficients $\wh \beta = (\wh \beta_1,\ldots, \wh \beta_n)^\T$ are obtained by  solving 
\[
\wh \beta =   \argmin_{\beta \in \RR^n}  \left \{
    {1\over n}\left\| Q_\mu^{1/2} (\YY - \sqrt{n} ~ \KK \beta ) \right\|^2
    +\lambda  ~ \beta^\T \KK \beta \right \}
\]
If $Q_\mu \KK+\lambda \II_n $ is invertible, then one can follow the same treatment as in \cref{app_pf_sec_2} to deduce
\[
  \wh \beta ~  =~ {1\over \sqrt n}\bigl(Q_\mu   \KK+\lambda \II_n\B)^{-1} Q_\mu  \YY,
\]
so that
\[
 \wh g_r(\XX) = \KK( Q_\mu \KK +\lambda \II_n )^{-1} Q_\mu \YY.
\]
Consequently, we have 
$
    \wh \alpha_r = ( \XX ^\T \XX+ n \mu  ~ \II_d) ^{-1}  \XX^\T (\YY -  \wh g_r(\XX)) .
$\\

It remains to prove that $Q_\mu  \KK+\lambda \II_n $ is invertible. When $\mu =0$, the proof is complete in \cref{app_pf_sec_2}. 
For $\mu >0$,  denote the eigenvalues of  $(1/n)\XX \XX^\T \in \RR^{n\times n}$ be $\wh \eta_1,\cdots,\wh \eta_n$, so that the eigenvalues of  $Q_\mu$ are
$1- \wh \eta_i/ (\wh \eta_i +\mu),~  i\in [n]$. Hence, $Q_\mu$ is positive  definite. 
Suppose that there exists a $v\in \RR^n$ such that
\[
(  Q_\mu  \KK + \lambda \II_n )~ v  = 0. 
\]
Then for $\lambda>0$, we have $v =  Q_\mu  u$ with $u = - \lambda^{-1} \KK v$, from which we find that 
\[
u^\T  Q_{\mu} \KK Q_\mu  u  +  \lambda u^\T Q_\mu^2  u=0,
\]
implying $u =0$.
This further gives
\[
 v =Q_\mu   u =0, 
\]
concluding that $ Q_\mu \KK + \lambda \II_n$ is invertible,
hence the proof is complete. 
\end{proof}

\begin{lemma}\label{lem_ridge}
For any $\mu, \lambda>0$, one has
   \[
\wh g_r
~ =~ \argmin_{ g \in \cH_K} \l\{
\frac{1}{n}~ \b\|Q_\mu ^{1/2}  (\YY-g(\XX)) \b\|^2 + \lambda\|g\|_K^2 
\r\}.
\]
\end{lemma}
\begin{proof}
  The objective function in \eqref{def_method_ridge} can be written as 
\[
Q (\alpha, g) ~ : =~
\frac{1}{n} \| \YY-\XX \alpha-g(\XX)\|^2+ \mu\|\alpha \| ^2+\lambda\|g\|_K^2. 
\]
Recalling  $\wh \alpha_{r,g} = (\XX^\T \XX + n\mu~  \II_d)^{-1}  (\YY-g(\XX)) $ from \eqref{def_alpha_r_g}, we have
\begin{align*}
 & Q(\wh \alpha_{r,g},f)\\ 
 ~=~  & \frac{1}{n} \|\YY - g(\XX)\|^2
- \frac{2}{n}\l \langle  \XX \wh \alpha_{r,g},  \YY-g(\XX)\r\rangle
+\frac{1}{n} \| \XX \wh \alpha_{r,g}\|^2
+ \mu \|\wh \alpha_{r,g}\|^2 +\lambda \|g\| _K^2
\\
~=~  &   \frac{1}{n} \|\YY- g(\XX)\|^2  -\frac{2}{n}\l \langle  \XX \wh \alpha_{r,g},  \YY-g(\XX)\r\rangle 
+ \frac{1}{n}  \wh \alpha_{r,g}^\T ( \XX^\T \XX + n \mu~ \II _d)~ \wh \alpha_{r,g} + \lambda\| g\| _K^2\\
~=~  &   \frac{1}{n} \|\YY- g(\XX)\|^2  -\frac{2}{n}(\YY-g(\XX))^\T P_\mu ~ (\YY-g(\XX))
+ \frac{1}{n}  \wh \alpha_{r,g}^\T ( \XX^\T \XX + n \mu~ \II _d)~ \wh \alpha_{r,g} + \lambda\| g\| _K^2.
\end{align*}
Using the fact  that
\[
(\XX^\T \XX+ n \mu ~ \II _d) ~ \wh \alpha_{r,g} ~=~  \XX^\T (\YY-g(\XX)),
\]
we obtain
\begin{align*}
     \wh \alpha_{r,g}^\T ( \XX^\T \XX + n \mu~ \II _d)~ \wh \alpha_{r,g}     &  ~=~ \wh \alpha_{r,g}^\T  \XX^\T (\YY-g(\XX))\\
     & ~=~  (\YY-g(\XX))^\T P_\mu ~ (\YY-g(\XX))&&\text{by \eqref{def_P_mu},}
\end{align*}
from which we have
\begin{align*}
    Q(\wh \alpha_{r,g},f) 
    ~=~
\frac{1}{n}\b \| Q_\mu^{1/2}(\YY-g(\XX) )\b\|^2 +\lambda \| g\|_K^2.
\end{align*}
The proof is complete. 
\end{proof}

\subsection{Proof of \texorpdfstring{\cref{thm_fix_ridge}}{\texttwoinferior}}\label{app_pf_thm_ridge}

    Recall  $P_\mu$ and $Q_{\mu}$ from \cref{def_P_mu}. 
    Recall that  $ \wh \Sigma = (1/n) \XX^\T\XX $ and its eigenvalues  are written as $\wh \tau_1\ge \cdots\ge\wh \tau_d\ge0$. 
For any $\lambda,\mu>0$, define
\begin{align}\label{def_P_lam_mu}
    P_{ \mu,\lambda}  ~ =~  \KK( Q_{\mu}  \KK+\lambda \II_n)^{-1} Q_\mu , \qquad Q_{\mu,\lambda}  ~ =~    \II_n - P_{ \mu,\lambda} .
\end{align}
Further define $\wt \VV$  and write its eigen-decomposition as 
\begin{equation}\label{V_eigen_wt_ridge}
 \wt    \VV :~=~Q_{\mu} ~ \KK Q_{\mu} ~  =~ \sum_{j=1}^n   \wt \nu_j  \wt v_j  \wt v_j^\T
\end{equation} 
The following lemma is the version of \cref{lem_rem_Q_PQ_lbd} adapted to the present setting. Its proof follows the same argument as that of \cref{lem_rem_Q_PQ_lbd} and is omitted here.
\begin{lemma}\label{lem_rem_Q_PQ_lbd_ridge}
    Let  $  P_{ \mu,\lambda}$ and $Q_{\mu,\lambda}$ be defined in \eqref{def_P_lam_mu}. One has 
    \begin{align}\label{eq_QX_Plbd_ridge} 
         Q_\mu     P_{ \mu,\lambda}  &~=~  Q_\mu \KK^{1/2} (\KK^{1/2}  Q_\mu \KK^{1/2} + \lambda \II_n )^{-1} \KK^{1/2} Q_\mu  ~=~    \wt \VV~  \wt \VV_\lambda^{-1},\\\label{eq_QX_Qlbd_ridge}
     Q_\mu     Q_{ \mu,\lambda}& ~=~    \lambda    Q_\mu (    Q_\mu \KK   Q_\mu +\lambda \II_n )^{-1}  Q_\mu ~=~   \lambda   Q_\mu ~\wt \VV_\lambda^{-1}   Q_\mu .
    \end{align}
    Immediately, we have $\|  Q_\mu P_{ \mu,\lambda}\|_\op\le 1 $ and $\|   Q_\mu  Q_{ \mu,\lambda}  \|_\op\le 1$.
\end{lemma}  
\medskip

\begin{proof}[Proof of \cref{thm_fix_ridge}]
Using \eqref{fit_ridge} to deduce
\begin{align*}
      \wh f_r (\XX) - f^*   (\XX )     
 & ~ =~      P_\mu \YY  +  Q_\mu  P_{ \mu,\lambda}~  \YY - f^*  (\XX)\\  \nonumber
& ~ =~     Q_\mu   \cE  +  Q_\mu P_{ \mu,\lambda}  \cE +    Q_\mu P_{ \mu,\lambda}  ~  f^*(\XX) -  Q_\mu f^*(\XX) \\
&~ = ~ P_ \mu  \cE  + Q_\mu P_{ \mu,\lambda}~  \cE -  Q_\mu Q_{\mu,\lambda}~ f^*(\XX),
\end{align*} 
we obtain 
\[
\cR(\wh  f_r  - f^* ) ~ \le ~ {\sigma^2\over n} ~  \tr\l( (P_ \mu  + Q_\mu P_{ \mu,\lambda}  ) ^\T  (P_ \mu +Q_\mu P_{ \mu,\lambda} )\r) 
+ 
{1\over n} \l\|  Q_\mu Q_{\mu,\lambda} f^*(\XX) \r \|^2 \\
=: \rI + \rII.  
\]
Regarding $\rI$,  we note that
\begin{equation}\label{can_be_repla}
    \begin{aligned}
  &  \tr\l( (P_ \mu  + Q_\mu P_{ \mu,\lambda}  ) ^\T  (P_ \mu +Q_\mu P_{ \mu,\lambda} )\r) \\
      &  ~     \le ~  2~ \tr\l(P_ \mu ^2 \r)  +2~   \tr\l( ( Q_\mu P_{ \mu,\lambda} )^2   \r)      \\ 
 & ~= ~ 
 2~  \tr\l(P_ \mu ^2 \r)   
 +  2~   \tr\l( \wt \VV ^2 ~  \wt \VV_\lambda^{-2}    \r)  &&\text{by \eqref{eq_QX_Plbd_ridge}}
 \\ 
 & ~ =~ 
 2~ \sum_{i=1 }^n \l( \frac{\wh \tau_i }{\wh \tau_i  + \mu }  \r)^2
 + 2~  \sum_{i=1 }^n \l( \frac{ \wt \nu_i  }{\wt \nu_i +\lambda} \r)^2&&\text{by \eqref{V_eigen_wt_ridge}},
    \end{aligned}
\end{equation}
which gives the upper bound of $\rI$. 
\\

Regarding $\rII$,  for any $\lambda,\mu>0$,
define 
\begin{align} \label{ridge_1}
f_{\lambda, \mu } 
: ~=~  \argmin_{ f \in \cH_K } \l\{ {1\over n} 
\l \|Q_\mu^{1/2} ( f(\XX) - f^*(\XX)) \r\|^2+\lambda \|f\|_K^2 \r\}, 
\end{align}
By the same reasoning of proving \eqref{closed_form_ridge}
with $f^*(\XX)$ in place of $\YY$, we have
\[
    f_{\lambda, \mu } ~=~ P_{\mu, \lambda} f^*(\XX),
\]
so that
\begin{align*}
\b\|  Q_\mu Q_{\mu,\lambda}  f^*(\XX)  \b \|^2 ~=~ & \l \|  Q_\mu    f^*(\XX) -  Q_\mu P_{\mu, \lambda}   f^*(\XX) \r\|^2  \\
~=~&  \l \|  Q_\mu    f^*(\XX) -  Q_\mu  f_{\lambda, \mu }(\XX ) \r\|^2 
\\
~ \le~ &  \l \| Q_\mu^{1/2}    f^*(\XX) -Q_\mu ^{1/2}  f_{\lambda, \mu }(\XX ) \r\|^2 &&\text{by $\|  Q_\mu \|_\op \le 1$}\\
~ \le~ &  \argmin_{ f \in \cH_K } \l\{ {1\over n} 
\l \|Q_\mu^{1/2} ( f(\XX) - f^*(\XX)) \r\|^2+\lambda \|f\|_K^2 \r\} 
&&\text{by \eqref{ridge_1}}  \\
~=~  & \argmin_{\alpha\in \RR^n,  f \in \cH_K } \l\{ {1\over n} 
\l \| \XX^\T \alpha + f(\XX) - f^*(\XX)) \r\|^2+\mu\|\alpha \|^2 +\lambda \|f\|_K^2 \r\},  
\end{align*}
where the last step holds by 
following the same argument in \cref{lem_ridge} with $f^*(\XX)$ in place of $\YY$. This gives the upper bound of $\rII$. 
Collecting the upper bounds of $\rI, \rII$ completes the proof.
\end{proof}

\subsection{Proof of \texorpdfstring{\cref{bd_pred_rate_random_ridge}}{\texttwoinferior}}\label{app_pf_thm_random_ridge}
\begin{proof} 
By noting that $Q_\mu \KK Q_\mu~\preceq~ \KK$ and 
using the same argument in the proof of \cref{lem_comp} with $Q_\mu$ in place of $Q_{\XX}$  gives
\[
 \sum_{i=1}^n  \left(\wt\nu_i \over \wt\nu_i + \lambda\right)^2~ \le  ~  \sum_{i=1}^n  \left(\wh\mu_i \over \wh\mu_i + \lambda \right)^2,\qquad\forall~\lambda>0  .
\]
Then in view of the argument of proving \cref{thm_random}, we have
    \[
    \EE  \tr\l( \wt \VV ^2 ~  \wt \VV_\lambda^{-2}    \r) ~\le ~  \EE  \tr\l(  \KK  \KK _\lambda^{-1}    \r) ,
    \]
    so that it remains to derive an upper bound for 
    $\EE [ \tr(P_\mu^2)]$.

    To this end, note that
    \[
 \wh \Sigma = {1\over n}  \XX^\T \XX   = {1\over n} \sum_{i=1}^n X_i X_i ^\T. 
    \]
We  find that
    \begin{align*}
        \EE[\tr(P_\mu^2) ]  ~\le ~        \|P_\mu\|_\op~\EE~  \tr(P_\mu )   ~\le ~  &    \EE~ \tr(P_\mu)  &&\text{by $\|P_\mu\|_\op\le 1$}  \\
~=  ~    &   \EE\l[\tr( \XX (\XX^\T\XX+n\mu \II_d)^{-1}\XX^\T)\r ]  
      \\
~=  ~     &   \EE\l[ \tr(  \wh \Sigma ~ ( \wh \Sigma +  \mu \II_d)^{-1} )\r ]   \\
~\le   ~     &  \tr\l(   \EE \wh \Sigma ~ (  \EE \wh \Sigma +  \mu \II_d)^{-1} \r) &&\text{by Jensen's inequality} \\ 
~=  ~     &  \tr( \Sigma ~ ( \Sigma +  \mu \II_d)^{-1} ) .
    \end{align*}
    The proof is complete. 
\end{proof}

\section{Complementary Simulations}\label{app_sec_sim_ridge}
In this section we conduct simulation studies to compare the prediction performance of the adaptive KRR (AKRR) in \cref{sec_method_krr},  its ridge-regularized variant (AKRR-ridge) in \cref{sec_ridge}, standard  KRR,  linear ridge regression (LRR),   and OLS.

We consider model \eqref{model} where  $f^*$  is set   as in \eqref{regre_function},
 $\epsilon \sim N(0,1)$, and $X$ 
 follows the  multivariate normal distribution $N(0,\Sigma)$ with $\Sigma$ defined in  \eqref{X_covariance}.
The parameter $\varrho$ in $\Sigma$   controls the overall correlation strength, and $s$ controls the rate at which the correlations decay across coordinates of $X$. 
We choose two pairs: $(\varrho, s) = (0.6, 1)$ and $(\varrho, s) = (0.9, 6)$.
The latter choice corresponds to a stronger dependence within the coordinates of $X$.   We 
vary $\alpha$   to examine its effect on the performance of different predictors. 
The sample size is set to $n=600$, and the dimension of $X$ is set to $d=50$.   The Gaussian kernel $K(x,x') =\exp(-\|x-x'\|^2/\gamma)$ with $\gamma>0$ is adopted in the experiment. 
All the tuning parameters in penalties, together with the bandwidth $\gamma$,  are selected jointly via cross-validation (CV).
When the median-adjusted $\gamma$ stated in \cref{sec_gaussian} is used, we only 
tune penalty parameters via CV. 
The obtained 
results are reported in Tables \ref{tab_app_1} and \ref{tab_app_2}. 
\begin{table}
\centering
\caption{Average performance for 
$(\varrho, s) = (0.9, 6)$
}
\label{tab_app_1} 
\small
\setlength{\tabcolsep}{7pt}
\begin{tabular}{lcccccc}
\toprule
$\alpha$  & 0 & 0.4 & 0.8 & 1.2 & 1.6 & 2 \\
\midrule
OLS          & 0.090 & 1.214 & 4.593 & 10.337 & 18.244 & 22.749 \\
LRR          & 0.018 & 1.047 & 4.172 & 9.441  & 16.688 & 20.422 \\
\midrule
KRR with CV         & 0.020 & 0.304 & 0.792 & 1.528  & 2.531  & 2.903 \\
AKRR  with CV       & 0.091 & 0.423 & 0.969 & 1.748  & 2.779  & 3.207 \\
AKRR-ridge with CV & 0.022 & 0.296 & 0.774 & 1.534  & 2.477  & 2.849 \\
\midrule
KRR   with Med $\gamma$       & 0.015 & 0.264 & 0.759 & 1.548 & 2.512 & 3.717 \\
AKRR  with Med $\gamma$       & 0.092 & 0.372 & 0.949 & 1.809 & 2.766 & 4.086 \\
AKRR-ridg with Med $\gamma$  & 0.021 & 0.275 & 0.739 & 1.547 & 2.470 & 3.696 \\
\bottomrule
\end{tabular}
\end{table}

A few remarks are made as follows. First, 
observing  Tables  \ref{tab_app_1} and \ref{tab_app_2},
the proposed 
AKRR-ridge method 
outperforms other methods across almost all settings, which coincides with our theoretical expectation in \cref{sec_ridge}.
Either KRR or linear procedures lead to less satisfactory performance in prediction.  
More specifically, in the case $(\varrho, r) = (0.9, 6)$, both OLS and LRR exhibit degraded performance, owing to the much stronger dependence among the coordinates of $X$. In this case, 
the performance of  KRR gets close to that of AKRR-ridge. 
One explanation is that when the correlation between the coordinates in $X$ becomes strong, the kernel matrix $\KK$  depends on $\XX$  through only a few dominant directions, thereby mitigating the curse of dimensionality of KRR. 
As a result, KRR has dominant advantages over linear procedures, which significantly narrows the gap in prediction performance between KRR and AKRR-ridge.
In the case 
$(\varrho, r) = (0.6, 1)$, however, 
KRR performs clearly worse than AKRR-ridge, due to the curse of dimensionality.  The performance of  OLS,  LRR, and AKRR gets close to that of AKRR-ridge, as linear procedures are already highly competitive in this setting.
Second, we compare the methods with and without the ridge penalty. 
For the case $(\varrho, r) = (0.9, 6)$, we observe evident benefits from incorporating a ridge penalty to penalize the linear strength: LRR  outperforms OLS, and AKRR-ridge outperforms AKRR. This is due to the stronger dependence among the coordinates of $X$, under which the added regularization helps reduce variance more substantially. This sharply differs from that for the case $(\varrho, r) = (0.6, 1)$ in Table \ref{tab_app_2}, 
where the much weaker dependence leads to a subtle impact of the additional ridge penalty on the prediction performance. 
\begin{table}
\centering
\caption{Average performance for 
$(\varrho, s) = (0.6, 1)$
}
\label{tab_app_2} 
\small
\setlength{\tabcolsep}{7pt}
\begin{tabular}{lcccccc}
\toprule
$\alpha$  & 0 & 0.4 & 0.8 & 1.2 & 1.6 & 2 \\
\midrule
OLS          & 0.091 & 0.219 & 0.599 & 1.206 & 2.107 & 3.261 \\
LRR       & 0.093 & 0.217 & 0.605 & 1.225 & 2.091 & 3.192 \\
\midrule
KRR  with CV        & 0.119 & 0.237 & 0.621 & 1.218 & 2.079 & 3.717 \\
AKRR   with CV     & 0.091 & 0.220 & 0.599 & 1.207 & 2.107 & 3.259 \\
AKRR-ridge  with CV   & 0.091 & 0.216 & 0.600 & 1.193 & 2.070 & 3.166 \\
\midrule
KRR-Med  with Med $\gamma$        & 0.232 & 0.352 & 0.759 & 1.410 & 2.334 & 3.584 \\
AKRR-Med   with Med $\gamma$      & 0.097 & 0.229 & 0.607 & 1.222 & 2.104 & 3.305 \\
AKRR-ridge-Med with Med  $\gamma$  & 0.096 & 0.223 & 0.604 & 1.222 & 2.088 & 3.186 \\
\bottomrule
\end{tabular}
\end{table}
Third, 
comparing the methods with CV-selected $\gamma$ and those with median-adjusted $\gamma$, we observe that, in the case $(\varrho, r) = (0.9, 6)$, 
the median-adjusted $\gamma$ yields satisfactory performance for KRR, AKRR, and AKRR-ridge, except when $\alpha =2$. In addition, the benefits of using the median-adjusted $\gamma$ over the CV-selected $\gamma$ become more evident as $\alpha$  decreases. Conversely, in the case $(\varrho, r) = (0.6, 1)$, 
KRR with median-adjusted $\gamma$ has clearly worse performance than KRR with  CV-selected $\gamma$, except when $\alpha =2$. Differently, our results suggest that
both AKRR and AKRR-ridge are robust to the choice of  $\gamma$ in this setting.

\end{document}